\documentclass[preprint,10pt]{elsarticle2}

\usepackage{datetime}

\usepackage{amscd}
\usepackage{amssymb,amsmath,amsthm}

\usepackage{multirow}

\usepackage[latin1]{inputenc}
\usepackage{latexsym}
\usepackage{graphicx}

\usepackage{stmaryrd}

\usepackage{hyperref}

\usepackage{cleveref}
\usepackage[justification=centering]{caption}

\usepackage{amscd}

\usepackage{mathptmx}
\usepackage{mathrsfs}
\usepackage{color}
\usepackage{xspace}
\usepackage{bussproofs}

\usepackage{tikz}

\usepackage{centernot}

\usepackage{rotating}

\usepackage[b]{esvect}

\DeclareSymbolFont{letters}{OML}{cmm}{m}{it}

\DeclareMathAlphabet{\mathcal}{OMS}{cmsy}{m}{n}

\usepackage{colortbl}

\usepackage{enumerate}

\newtheorem{theorem}{Theorem}[section]
\newtheorem{defn}[theorem]{Definition}

\newtheorem{lem}[theorem]{Lemma}
\newtheorem{prop}[theorem]{Proposition}
\newtheorem{fact}[theorem]{Fact}
\newtheorem{exmp}[theorem]{Example}

\newenvironment{clm}[1][]{%
\par\addvspace{6pt}
\noindent
\textbf{Claim #1:}
\noindent
}
{\par\addvspace{4pt}}

\newenvironment{proofclaim}[1][]{%
\par
\noindent
\emph{Proof of Claim #1.}
\noindent
}
{\hfill$\dashv$\par\addvspace{6pt}}

\usepackage{array}
\newcolumntype{C}[1]{>{\centering\arraybackslash}p{#1}}

\newcommand{\D}{\ensuremath{\mathcal{D}}\xspace}

\newcommand{\Ind}{\ensuremath{\mathcal{I}}\xspace}

\newcommand{\dep}{\ensuremath{\mathop{=\!}}\xspace}

\newcommand{\sor}{\ensuremath{\cor}\xspace}

\newcommand{\cn}{\ensuremath{\mathop{\dot\sim}}\xspace}

\newcommand{\cor}{\ensuremath{\vee}\xspace}
\newcommand{\ior}{\ensuremath{\mathbin{\rotatebox[origin=c]{-90}{$\geqslant$}}}\xspace}
\newcommand{\bigior}{\ensuremath{\mathbin{{\large\rotatebox[origin=c]{-90}{$\geqslant$}}}}\xspace}

\newcommand{\logicFont}[1]{\mathcal{#1}}

\newcommand{\halfliteral}[1]{\protect\ensuremath{#1}}
\newcommand{\literal}[1]{\halfliteral{#1}\xspace}

\newcommand{\set}[3][]{\literal{\left\{#2\;\middle|\;\ifthenelse{\equal{#1}{}}{\text{#3}}{\parbox{#1}{#3}}\right\}}}

\DeclareMathSymbol{\Gamma}{\mathalpha}{operators}{0}
\DeclareMathSymbol{\Delta}{\mathalpha}{operators}{1}
\DeclareMathSymbol{\Theta}{\mathalpha}{operators}{2}
\DeclareMathSymbol{\Lambda}{\mathalpha}{operators}{3}
\DeclareMathSymbol{\Xi}{\mathalpha}{operators}{4}
\DeclareMathSymbol{\Pi}{\mathalpha}{operators}{5}
\DeclareMathSymbol{\Sigma}{\mathalpha}{operators}{6}
\DeclareMathSymbol{\Upsilon}{\mathalpha}{operators}{7}
\DeclareMathSymbol{\Phi}{\mathalpha}{operators}{8}
\DeclareMathSymbol{\Psi}{\mathalpha}{operators}{9}
\DeclareMathSymbol{\Omega}{\mathalpha}{operators}{10}

\newcommand{\logic}[1]{\literal{\logicFont{#1}}}
\newcommand{\paraLogic}[2]{\ensuremath{\logic{#1}\ifthenelse{\equal{#2}{}}{}{(#2)}}\xspace}

\newcommand{\LL}{\ensuremath{\mathcal{L}}\xspace}

\journal{arXiv.org}

\begin{document}

\begin{frontmatter}

\title{Negation and partial axiomatizations of dependence and independence logic revisited\tnoteref{tn}}

 \author{Fan Yang}

\address{Department of Values, Technology and Innovation, Delft University of Technology, Jaffalaan 5, 2628 BX Delft, The Netherlands}

\address{Department of Mathematics and Statistics, PL 68 (Pietari Kalmin katu 5), 00014 University of Helsinki, Finland}

\ead{fan.yang.c@gmail.com}

\tnotetext[tn]{This research was partially supported by grant 308712 of the Academy of Finland.}

\begin{abstract}
In this paper, we axiomatize the negatable consequences in dependence and independence logic by extending the systems of natural deduction of the logics given in \cite{Axiom_fo_d_KV} and \cite{Hannula_fo_ind_13}. We prove a characterization theorem for negatable formulas in independence logic and negatable sentences in dependence logic, and identify an interesting class of formulas that are negatable in independence logic. Dependence and independence atoms, first-order formulas belong to this class. We also demonstrate our extended system of independence logic by giving explicit derivations for Armstrong's Axioms and the Geiger-Paz-Pearl axioms of dependence and independence atoms. 
\end{abstract}

\begin{keyword}
dependence logic \sep team semantics \sep negation \sep existential second-order logic


\MSC[2010] 03B60
\end{keyword}

\end{frontmatter}



\section{Introduction}

 
 
 \emph{Dependence logic}  was introduced by V\"{a}\"{a}n\"{a}nen  \cite{Van07dl}  as a development of \emph{Henkin quantifiers}  \cite{henkin61}  and \emph{independence-friendly logic}  \cite{Hintikka98book}. Recently, Gr\"{a}del and V\"{a}\"{a}n\"{a}nen \cite{D_Ind_GV} defined a variant of dependence logic, called \emph{independence logic}. 
 The two logics add to first-order logic new types of atomic formulas $\dep(\vec{x},y)$ and $\vec{x}\perp_{\vec{z}}\vec{y}$, called \emph{dependence atom} and \emph{independence atom}, to explicitly specify the dependence and independence relations between variables. Intuitively, $\dep(\vec{x},y)$ states that ``the value of $y$ is completely determined by the values of the variables in the tuple $\vec{x}$\,'', and $\vec{x}\perp_{\vec{z}}\vec{y}$ states that ``given the values of the variables $\vec{z}$, the values of $\vec{x}$ and the values of $\vec{y}$ are completely independent of each other".  These properties cannot be meaningfully manifested in \emph{single} assignments of the variables. Therefore unlike  the usual Tarskian semantics where satisfaction relation is defined with respect to {\em single} assignments of a model, formulas of dependence and independence logic are evaluated on \emph{sets} of assignments (called \emph{teams}) instead. This semantics, introduced by Hodges \cite{Hodges1997a,Hodges1997b}, is called \emph{team semantics}.

Dependence and independence logic are known to have the same expressive power as existential second-order logic $\Sigma^1_1$ (see \cite{KontVan09} and \cite{Pietro_I/E}). 
This fact has two negative consequences: The logics are not closed under classical negation and they are not (effectively) axiomatizable. The aim of this paper is to shed some new light on these problems.

 Regarding the first problem, ``negation'', being usually a desirable connective for a logic, turns out to be a tricky connective in the context of team semantics. The negation that dependence and independence logic inherit from  first-order logic (denoted by  $\neg$) is a type of ``syntactic negation", in the sense that in order to compute the meaning of the formula $\neg\phi$, the negation $\neg$ has to be brought to the very front of the atomic formulas by applying  De Morgen's laws and the double negation law. It was proved  that this negation $\neg$ is actually not a semantic operator \cite{Negation_D_KV}, meaning  that the semantic equivalence of $\phi$ and $\psi$ does not necessarily imply that $\neg\phi$ and $\neg\psi$ are semantically equivalent. The \emph{classical (contradictory) negation} (denoted by $\sim$ in the literature), on the other hand, is  indeed a semantic operator. Whereas, neither dependence nor independence logic is closed under classical negation, as the $\Sigma^1_1$ fragment of second-order logic is not. 
 Dependence logic extended with the classical negation $\sim$ is called \emph{team logic} in the literature, and it was proved \cite{KontinenNurmi2011} to have the same expressive power as full second-order logic. 

Every formula $\phi$ of dependence and independence logic is satisfied by the empty team, which in turn cannot satisfy  the classical (contradictory) negation $\sim\phi$ of the formula. This  implies that $\sim\phi$  cannot possibly be definable in dependence or independence logic for any  formula $\phi$. This technical subtlety makes the classical contradictory negation $\sim$ less interesting in the context of dependence and independence logic. In this paper, we will, instead, consider the \emph{weak classical negation}, denoted by \cn, which behaves exactly as the classical negation  except that on the empty team $\cn\phi$ is always satisfied. We call a formula $\phi$ {\em negatable} in a logic if its weak classical negation $\cn\phi$ is definable in the logic. We will prove a characterization theorem for negatable formulas in independence logic and negatable sentences in dependence logic by generalizing an argument in \cite{Van07dl}, from which it will follow that the problem of determining whether a formula  is negatable in independence logic or in dependence logic is undecidable. Yet, we will identify an interesting class of formulas, presented as a hierarchy, that are negatable in independence logic. First-order formulas, dependence and independence atoms belong to this hierarchy.  Formulas of this class are closely related to the \emph{dependency notions} considered in \cite{Galliani2014SFOD} and the \emph{generalized dependence atoms} studied in \cite{Kuusisto2015,KontinenMullerSchnoorVollmer2014}.







As for the axiomatization problem, since $\Sigma^1_1$ is not (effectively) axiomatizable,  dependence and independence logic cannot possibly be axiomatized in full. Nevertheless, \cite{Axiom_fo_d_KV} and \cite{Hannula_fo_ind_13} defined systems of natural deduction for the logics such that the equivalence
\begin{equation}\label{cmp_eq}
\Gamma\models\phi\iff \Gamma\vdash\phi
\end{equation}
holds if $\Gamma$ is a set of sentences of dependence logic (or independence logic) and $\phi$ is a first-order sentence. It was left open whether these partial axiomatizations can be generalized 
such that the above equivalence (\ref{cmp_eq}) holds if $\Gamma$ is a set of formulas  (that possibly contain free variables) and $\phi$ is a  first-order formula (that possibly contain free variables too). Kontinen  \cite{Kontinen15foc} gave such a generalization by expanding the signature with an extra relation symbol so as to interpret the teams associated with  the free variables.
In this paper, we will generalize the partial axiomatization results in \cite{Axiom_fo_d_KV,Hannula_fo_ind_13} via a different approach, an approach that makes essential use of the weak classical negation. We extend the systems given in \cite{Axiom_fo_d_KV,Hannula_fo_ind_13} by adding an elimination rule for the weak classical negation (\cn\!\textsf{E}), and show that  the equivalence (\ref{cmp_eq}) holds for the extended systems if $\Gamma$ is a set of formulas and $\phi$ is a formula that is negatable in the logics. It should be noted that because the full class of negatable formulas is undecidable, as mentioned above, the weak classical negation elimination rule \cn\!\textsf{E} and thus the full extended systems are undecidable. Nevertheless, by restricting the rule \cn\!\textsf{E} to decidable subclasses of negatable formulas, such as the above-mentioned hierarchy of negatable formulas, our systems can already have interesting applications.  As an illustration, we will give in this paper explicit derivations of Armstrong's Axioms  \cite{Armstrong_Axioms} and the Geiger-Paz-Pearl axioms  \cite{Geiger-Paz-Pearl_1991} in database theory in our extended system of independence logic, and we will also demonstrate that  Arrow's Impossibility Theorem \cite{Arrow51} in social choice theory is derivable in  the system too. 



This paper is organized as follows. In Section 2 we recall the basics of dependence and independence logic. Section 3 proves a characterization theorem for negatable formulas or sentences in the logics. In Section 4 we extend the systems of natural deduction of dependence and independence logic in  \cite{Axiom_fo_d_KV,Hannula_fo_ind_13}  to axiomatize negatable consequences in the logics. Section 5 identifies a class of negatable formulas in independence logic, of which dependence and independence atoms are members. In Section 6 we illustrate the extended system of independence logic by deriving Armstrong's Axioms and the Geiger-Paz-Pearl axioms of dependence and independence atoms in the system. We finish by making some concluding remarks in Section 7.

%
%

\section{Preliminaries}

Let us start by recalling the syntax and semantics (i.e. team semantics) of dependence and independence logic. 

We first fix the syntax of first-order logic.
Given a signature $\mathcal{L}$. First-order 
atomic $\mathcal{L}$-formulas $\lambda$ are defined as usual. Well-formed $\mathcal{L}$-formulas of first-order logic, also called \emph{first-order formulas}, (in \emph{negation normal form}) are defined by the following grammar:
\[\phi::=\lambda\mid\neg\lambda\mid\bot\mid \phi\wedge\phi\mid \phi\cor\phi\mid \exists x\phi\mid \forall x\phi\]
where $\lambda$ is an arbitrary first-order atomic $\mathcal{L}$-formula. 

Formulas will be evaluated on  the usual first-order models over the signature $\mathcal{L}$. We will use the same notation $M$ for both a model and its domain, and assume that $M$ has at least two elements. 
We write $\mathcal{L}(R)$ for the signature expanded from $\mathcal{L}$ by adding a fresh relation symbol $R$, and $(M,R^M)$ denotes the $\mathcal{L}(R)$-expansion of $M$ in which the $k$-ary relation symbol $R$ is interpreted as $R^M\subseteq M^k$. We write $\phi(R)$ to emphasize that the relation symbol $R$ occurs in  the formula $\phi$.


Although team semantics is intended for extensions of first-order logic obtained by adding dependence or independence atoms, for the sake of comparison we will now introduce the team semantics for first-order logic too. A \emph{team} $X$ of a model $M$ over a set $V$ of first-order variables is a set of assignments of $M$ over $V$, i.e., a set of functions $s:V\to M$. The set $V$  is called the \emph{domain} of $X$, denoted by  ${\rm dom}(X)$.


There is one and only one assignment of $M$ over the empty domain, namely the empty assignment $\emptyset$. The singleton of the empty assignment $\{\emptyset\}$ is a team of $M$, and the empty set $\emptyset$ is a team of $M$ over any domain. 

Let $s$ be an assignment of $M$ over $V$ and $a\in M$. We write $s(a/x)$ for the assignment of $M$ over $V\cup\{x\}$ defined as $s(a/x)(x)=a$ and $s(a/x)(y)=s(y)$ for all $y\in V\setminus \{x\}$. For any set $N\subseteq M$ and any function $F:X\to \wp(M)\setminus\{\emptyset\}$, define 
\[X(N/x)=\{s(a/x): a\in N,~s\in X\}\] and \[X(F/x)=\{s(a/x):s\in X\text{ and }a\in F(s)\}.\]
We write  simply $X(a/x)$ for $X(\{a\}/x)$.  Denote by $\vec{x}$ or $\textsf{x}$  a sequence $x_1,\dots,x_n$ of variables and the length $n$ will always be clear from the context or does not matter; similarly for a sequence $\vec{F}$ of functions and a sequence $\vec{s}$ of assignments. A team $X(N/x_1)\dots (N/x_n)$ will sometimes be abbreviated as $X(N/\vec{x})$, and a team $X(F_1/x_1)\dots(F_n/x_n)$ as $X(F_1/x_1,\dots,F_n/x_n)$ or 
$X(\vec{F}/\vec{x})$.

We now define the team semantics for first-order formulas. Note that our version of the team semantics for  disjunction and  existential quantifier is known as the \emph{lax semantics} in the literature (see e.g., \cite{Pietro_I/E}).
\begin{defn}\label{TS_FO}
Define inductively  the notion of a first-order formula $\phi$ being \emph{satisfied} on a model $M$ and a team $X$, denoted by $M\models_X\phi$, as follows:
\begin{itemize}
\item $M\models_X\lambda$ with $\lambda$ a first-order atomic formula ~iff~
for all $s\in X$, $M\models_s\lambda$ in the usual sense.
  \item $M\models_X\neg\lambda$ with $\lambda$ a first-order atomic formula ~iff~
for all $s\in X$, $M\not\models_s\lambda$ in the usual  sense.
  \item $M\models_X\bot$ ~iff~ $X=\emptyset$.
  \item $M\models_X\phi\wedge\psi$ ~iff~ $M\models_X\phi$ and
  $M\models_X\psi$.
  \item $M\models_X\phi\sor\psi$ ~iff~ there exist $Y,Z\subseteq X$ with $X=Y\cup Z$ such that 
  $M\models_Y\phi$ and $M\models_Z\psi$.
   \item $M\models_X\exists x\phi$ ~iff~ $M\models_{X(F/x)}\phi$ for some function $F:X\to
  \wp(M)\setminus\{\emptyset\}$.
  \item $M\models_X\forall x\phi$ ~iff~ $M\models_{X(M/x)}\phi$.
\end{itemize}
\end{defn}

A routine inductive proof shows that first-order formulas have the downward closure property and the union closure property:
\begin{description}
\item[(Downward Closure Property)] $M\models_X\phi$ and $Y\subseteq X$ imply $M\models_Y\phi$.

\vspace{4pt}
\item[(Union Closure Property)] $M\models_{X_i}\phi$ for all $i\in I$ implies $M\models_{\bigcup_{i\in I}X_i}\phi$.
\end{description}
which combined are equivalent to the flatness property:
\begin{description}
\item[(Flatness Property)] \(M\models_X\phi\iff M\models_{\{s\}}\phi\text{ for all }s\in X\).
\end{description}
It is easy to see that over singleton teams the team semantics for first-order formulas $\phi$ coincides with the usual single-assignment semantics in the  sense that
\begin{equation}\label{fo_team_single_semantics}
M\models_{\{s\}}\phi\iff M\models_s\phi
\end{equation}
holds for any  model $M$ and  assignment $s$. If $\phi$ is a first-order formula, then the string $\neg\phi$, called the \emph{syntactic negation} of $\phi$, can be viewed as a first-order formula in negation normal form obtained in the usual way (i.e. by applying  De Morgan's laws, the double negation law, etc.), and we write $\phi\to\psi$ for the  formula $\neg\phi\cor\psi$. 
Since the law of excluded middle $\phi\cor\neg\phi$ holds for first-order formulas
 under the usual single-assignment semantics, by Equivalence (\ref{fo_team_single_semantics}) we know that $\mathop{M\models_{\{s\}}\phi\cor\neg\phi}$ always holds, which, together with the flatness property,  implies that $M\models_X\phi\cor\neg\phi$ holds for all teams $X$ and  models $M$. In other words, the law of excluded middle holds for first-order formulas also in the sense of team semantics.

We now turn to dependence and independence logic. 
Well-formed $\mathcal{L}$-formulas of independence logic (\Ind) are defined by the following grammar:
\[\begin{split}
\phi::=~& \lambda\mid\neg\lambda\mid \bot\mid x_1\dots x_n\perp_{z_1\dots z_k} y_1\dots y_m\mid \dep(x_1,\dots, x_n,y)\mid x_1\dots x_n\subseteq y_1\dots y_n\mid\\
&  \phi\wedge\phi\mid \phi\cor\phi\mid \exists x\phi\mid \forall x\phi
\end{split}\]
where $\lambda$ ranges over first-order atomic $\mathcal{L}$-formulas. The formulas $\dep(\vec{x},y)$, $\vec{x}\perp_{\vec{z}}\vec{y}$ and $\vec{x}\subseteq \vec{y}$ are called \emph{dependence atom}, \emph{independence atom} and \emph{inclusion atom}, respectively.  We refer to any of these atoms as \emph{atoms of dependence and independence}. 
 For the convenience of our argument in the paper, the independence logic as defined has a richer syntax than the standard one in the literature, which has the same syntax as first-order logic extended with independence atoms only. The other atoms 
are definable in the standard independence logic; for a proof, see e.g., \cite{Pietro_thesis}. \emph{Dependence logic} (\D), which is a fragment of \Ind, is defined as first-order logic extended with dependence atoms, and first-order logic extended with inclusion atoms is called \emph{inclusion logic}. In this paper we will only concentrate on dependence logic and independence logic.



The set $\mathop{\rm{Fv}}(\phi)$
 of free variables of a formula $\phi$ of \Ind is defined as usual and we also have the new cases for dependence and independence atoms:
\begin{itemize}
\item ${\rm Fv}(x_1\dots x_n\perp_{z_1\dots z_k} y_1\dots y_m)=\{x_1,\dots,x_n,y_1,\dots,y_m,z_1,\dots,z_k\}$,
\item ${\rm Fv}(\dep(x_1,\dots,x_n,y))=\{x_1,\dots,x_n,y\}$,
\item ${\rm Fv}(x_1,\dots,x_n\subseteq y_1,\dots,y_n)=\{x_1,\dots,x_n,y_1,\dots,y_n\}$.
\end{itemize}
We write $\phi(\vec{x})$ to indicate that the free variables occurring in $\phi$ are among $\vec{x}$. A formula $\phi$ is called a \emph{sentence} if it has no free variable.  

We write $\phi(t/x)$ for the formula obtained from $\phi$ by substituting uniformly every free occurrence of the variable $x$ for the term $t$. We sometimes abbreviate the substituted formula $\phi(t_1/x_1)\cdots(t_n/x_n)$ as $\phi(\vec{t}/\vec{x})$, and for a formula $\phi(\vec{x})$ and a list $\vec{w}$ of fresh variables we sometimes also write $\phi(\vec{w})$ for $\phi(\vec{w}/\vec{x})$. We write $s(\vec{x})$ for $\langle s(x_1),\dots,s(x_n)\rangle$.

\begin{defn}\label{TS_Ind}
Define inductively  the notion of a formula $\phi$ of \Ind being \emph{satisfied} on a model $M$ and a team $X$, denoted by $M\models_X\phi$. All the cases are identical to those defined in  \Cref{TS_FO} and additionally:
\begin{itemize}
\item $M\models_X\vec{x}\perp_{\vec{z}}\vec{y}$  ~iff~
for all $s,s'\in X$, $s(\vec{z})=s'(\vec{z})$ implies that there exists $ s''\in X$ such that \vspace{-0.5\baselineskip}
\[s''(\vec{z})=s(\vec{z})=s'(\vec{z}),~s''(\vec{x})=s(\vec{x})\text{ and }s''(\vec{y})=s'(\vec{y}).\]
  \item $M\models_X\dep(\vec{x},y)$ ~iff~ for all $s,s'\in X$, 
$s(\vec{x})=s'(\vec{x})$ implies $s(y)=s'(y)$.
    \item $M\models_X\vec{x}\subseteq\vec{y}$  ~iff~ for all $s\in X$, there exists $s'\in X$ such that $s'(\vec{y})=s(\vec{x})$.
 \end{itemize}
\end{defn}
We write $\vec{x}\perp\vec{y}$ for $\vec{x}\perp_{\langle\rangle}\vec{y}$, and $\dep(x)$ for $\dep(\langle\rangle,x)$, where $\langle\rangle$ is the empty sequence. Note that the semantic clauses for $\vec{x}\perp\vec{y}$ and $\dep(x)$ reduce to
\begin{itemize}
\item {\em $M\models_X\vec{x}\perp\vec{y}$  ~iff~ for all $s,s'\in X$, there exist $s''\in X$ such that 
\[s''(\vec{x})=s(\vec{x})\text{ and }s''(\vec{y})=s'(\vec{y}).\]}
\item {\em $M\models_X\dep(x)$  ~iff~ for all $s,s'\in X$, $s(x)=s'(x)$.}
\end{itemize}


A sentence $\phi$ is said to be \emph{true} in $M$, written $M\models\phi$, if 
$M\models_{\{\emptyset\}}\phi$. We write  $\Gamma\models\psi$
if for any model $M$ and  team $X$,
$M\models_X \phi\text{ for all }\phi\in\Gamma$ implies $M\models_X \psi$.
We also write $\phi\models\psi$ for $\{\phi\}\models\psi$. If $\phi\models\psi$ and $\psi\models\phi$, then we write $\phi\equiv\psi$.


We leave it for the reader to verify that formulas of dependence logic have the downward closure property, 
and formulas of independence logic have the empty team property and the locality property defined below:
\begin{description}
\item[(Empty Team Property)] $M\models_\emptyset\phi$ 

\vspace{4pt}

\item[(Locality Property)] If $\{s\upharpoonright {\rm Fv}(\phi)\mid s\in X\}=\{s\upharpoonright {\rm Fv}(\phi)\mid s\in Y\}$\footnote{For an assignment $s:V\to M$ and a set $V'\subseteq V$ of variables, we write $s\upharpoonright V'$ for the restriction of $s$ to the domain $V'$.}, then 
\[M\models_X\phi\iff M\models_Y\phi.\]
\end{description}

Recall that the existential second-order logic ($\Sigma^1_1$) consists of those formulas that are equivalent to some formulas of the form
\(\exists R_1\dots\exists R_k\phi,\) 
where $\phi$ is a first-order formula. An $\LL(R)$-sentence $\phi(R)$ of $\Sigma^1_1$ is said to be \emph{downward monotone} with respect to $R$ if 
$(M,Q)\models\phi(R)\text{ and }Q'\subseteq Q$ imply $(M,Q')\models\phi(R)$.
It is known that $\phi(R)$ is downward monotone with respect to $R$ if and only if $\phi(R)$ is equivalent to a sentence where $R$ occurs  only negatively (see e.g., \cite{KontVan09}).
A team $X$ of $M$ over $\{x_1,\dots,x_n\}$ induces an $n$-ary relation
\[rel(X):=\{(s(x_1),\dots,s(x_n))\mid s\in X\}\]
 on $M$; conversely, an $n$-ary relation $R^M$ on $M$ induces a team
\[X_R:=\{\{(x_1,a_1),\dots,(x_n,a_n) \}\mid (a_1,\dots,a_n)\in R^M\}.\]

\begin{theorem}[see \cite{Van07dl,KontVan09,Pietro_I/E}]\label{ind2sigma11}
\begin{description}
\item[(i)] Every $\LL$-sentence $\phi$  of \Ind (or \D) is equivalent to an $\LL$-sentence $\tau_\phi$ of $\Sigma^1_1$, i.e., 
\[M\models\phi\iff M\models\tau_\phi\]
holds for any $\LL$-model $M$; and conversely, every $\LL$-sentence $\psi$ of $\Sigma^1_1$ is equivalent to an $\LL$-sentence $\rho(\psi)$  of \Ind (or \D).
\item[(ii)] For every \LL-formula $\phi$  of   \Ind, there is an $\LL(R)$-sentence $\tau_\phi(R)$ of $\Sigma^1_1$ such that for all $\LL$-model $M$ and teams $X$,
\[M\models_X\phi\iff (M,rel(X))\models\tau_\phi(R).\]
If, in particular, $\phi$ is a formula   of \D, then the relation symbol $R$ occurs in the sentence $\tau_\phi(R)$ only negatively. 
\item[(iii)] For every $\LL(R)$-sentence $\psi(R)$ of $\Sigma^1_1$  that is downward monotone with respect to $R$, there is an \LL-formula $\rho(\psi)$  of \D  such that for all $\LL$-models $M$ and teams $X$,
\begin{equation}\label{sig11_ind_d}
M\models_X\rho(\psi)\iff (M,rel(X))\models\psi(R)\vee\forall \vec{x}\neg R\vec{x}.
\end{equation}
\item[(iv)] For every $\LL(R)$-sentence $\psi(R)$ of $\Sigma^1_1$, there is an \LL-formula $\rho(\psi)$  of \Ind  such that (\ref{sig11_ind_d}) holds for all $\LL$-models $M$ and teams $X$.
\end{description}
\end{theorem}
In the sequel, we will use the notations $\tau_\phi$ and $\tau_\phi(R)$ to denote the (up to semantic equivalence) unique formulas obtained in the above theorem and refer to them as the \emph{$\Sigma^1_1$-translations} of the formulas $\phi$ of \D or \Ind.


\section{First-order formulas and negatable formulas}\label{sec:negatable_frm}



Formulas of dependence and independence logic can be translated into $\Sigma^1_1$ (\Cref{ind2sigma11}). We thus use the term ``first-order formula'' in the team semantics setting in two senses:
A first-order formula $\phi$ can be viewed either  as a formula of \D or \Ind that is to be evaluated on teams, or as a usual  formula of first-order logic that is to be evaluated on single assignments and is possibly (equivalent to) the $\Sigma^1_1$-translation $\tau_\psi$ of some formula $\psi$  of \D or \Ind. With the latter reading, $M\models_s\neg\phi$ iff $M\not\models_s\phi$ holds for all models $M$ and assignments $s$, and the formula $\neg\phi$ is thus understood as the classical (contradictory) negation of $\phi$. However, on the team semantics side, unless the team $X$ is a singleton, $M\not\models_X\phi$ is in general not equivalent to $M\models_X\neg\phi$. To express the contradictory negation in the team semantics setting, let us define the \emph{classical negation} $\sim$ and the \emph{weak classical negation} $\cn$  as follows:
\begin{itemize}
\item {\em $M\models_X\sim\phi$ ~iff~ $M\not\models_X\phi$,}
\item {\em $M\models_X\cn\phi$ ~iff~ either $M\not\models_X\phi$ or $X=\emptyset$.}
\end{itemize}
Since formulas of dependence and independence logic have the empty team property, the classical negation $\sim\phi$ of any formula $\phi$ is not definable in the logics and we are therefore not interested in the classical negation $\sim$ in this paper. On the other hand, the weak classical negation $\cn\phi$  can be definable in the logics for some formulas $\phi$. We say that a formula $\phi$ of \Ind (or \D) is \emph{negatable}  if there is a formula $\psi$ of \Ind (or \D) such that $\cn\phi\equiv\psi$.  In this case, we also say that $\phi$ is negatable in \Ind (or \D).

For any first-order sentence $\phi$,  we have $M\not\models_{\{\emptyset\}}\phi$ iff $M\models_{\{\emptyset\}}\neg\phi$ by the law of excluded middle. Thus $\cn\phi\equiv\neg\phi$, meaning that first-order sentences  are negatable both in \D and in \Ind. Recall that first-order formulas are flat, and we now prove that negatable formulas of \D are, actually, all flat. 

\begin{fact}
If  a formula $\phi$ of \D is negatable, then it is  upward closed (i.e. $M\models_X\phi$ and $\emptyset\neq X\subseteq Y$ imply $M\models_Y\phi$), and thus flat.
\end{fact}
\begin{proof}
Suppose $\phi$ is a formula of \D that is not upward closed. Then, there exist a model $M$ and two teams $X\neq\emptyset$ and $Y\supseteq X$ such that $M\models_X\phi$ and $M\not\models_Y\phi$. 
But this means that $\cn\phi$ is not downward closed and thus not definable in \D.
\end{proof}




We will see next that the above fact does not apply to independence logic. Also note that sentences are always upward closed (since a sentence is true either on all teams or on the empty team  only). Thus, the converse direction of the above fact, if true, would imply that all sentences  of \D are negatable. But this is not the case, as we will see in the following characterization theorem for negatable sentences in \D and negatable formulas in \Ind.  



\begin{theorem}\label{neg_charac_thm_ind}
\begin{description}
\item[(i)] An $\LL$-formula $\phi$  of  \Ind is negatable if and only if its $\Sigma^1_1$-translation $\tau_\phi(R)$  is equivalent to a first-order sentence.
\item[(ii)] An \LL-sentence $\phi$  of \D is negatable if and only if its $\Sigma^1_1$-translation $\tau_\phi$ is equivalent to a first-order sentence.
\end{description}
\end{theorem}






Before we give the proof the  above theorem, let us make two remarks. First, the theorem states that negatable formulas in \Ind are exactly those formulas that have (essentially) first-order translations, and negatable sentences in \D are exactly those sentences that have (essentially) first-order translations. Therefore the problem of determining whether a formula of  \Ind or a sentence of \D is negatable reduces to the problem of determining whether a $\Sigma^1_1$-sentence ($\tau_\phi$) is equivalent to a first-order formula, or whether the second-order quantifiers in a $\Sigma^1_1$-sentence can be eliminated. This problem 
is known to be undecidable (this follows from e.g., \cite{Chagrova91}).



Next, we make some comments on yet another type of negation, the \emph{intuitionistic negation}.  Abramsky and V\"{a}\"{a}n\"{a}nen \cite{AbVan09} introduced the \emph{intuitionistic implication} (denoted by $\to$ too) that has the semantics clause:
\begin{itemize}
\item {\em $M\models_X\phi\to\psi$ ~iff~ for all $Y\subseteq X$,  $M\models_Y\phi$ implies $M\models_Y\psi$.}
\end{itemize}
It is easy to check that if $\phi$ and $\psi$ are first-order formulas, then $\phi\to\psi\equiv\neg\phi\vee\psi$ (and thus our slight abuse of the notation $\to$ does not cause any essential problem).
The \emph{intuitionistic negation} of a formula $\phi$  is defined (according to the usual convention) as $\phi\to\bot$ and its semantics clause reduces to:
\begin{itemize}
\item {\em $M\models_X\phi\to\bot$ ~iff~ for all nonempty $Y\subseteq X$,  $M\not\models_Y\phi$.}
\end{itemize}
If a formula $\phi$ of $\D$ is negatable, then both $\phi$ and $\cn\phi$ are downward closed, and thus $\cn\phi\equiv\phi\to\bot$. Also, for sentences the weak classical negation and the intuitionsitic negation coincide, i.e., $\cn\phi\equiv\phi\to\bot$ holds for $\phi$ being a sentence as well. However, the intuitionistic negation $\phi\to\bot$ is not in the language of \D or \Ind, and in fact, its second-order translation is in general a $\Pi_1^1$-sentence, which is in general not expressible in \D or \Ind (see \cite{AbVan09,Yang2011} for details). For an arbitrary formula $\phi$, the intuitionistic negation $\phi\to\bot$ is not necessarily equivalent to the weak classical negation $\cn\phi$; for instance, $\dep(x)\to\bot\equiv\bot\not\equiv\cn\dep(x)$. 

Let us now turn to the proof of \Cref{neg_charac_thm_ind}.
Item (ii) actually follows implicitly from the results in \cite{Van07dl}, and item (i) can be proved by essentially the argument of Theorem 6.7  in  \cite{Van07dl}. 
For the convenience of the arugment, let us first direct our attention to the $\Sigma^1_1$ counterpart of dependence and independence logic and prove a general theorem for $\Sigma^1_1$. 


\begin{theorem}\label{neg_char_thm_so}
\begin{description}
\item[(i)] Let $\phi(R)$ be an $\LL(R)$-formula of $\Sigma^1_1$ such that $(M,\emptyset)\models\phi(R)$ for any $\LL$-model $M$. The formula $\neg\phi\vee\forall \vec{x}\neg R\vec{x}$ belongs to $\Sigma^1_1$  if and only if $\phi$ is equivalent to a first-order formula. 
\item[(ii)] Let $\phi$ be an \LL-formula of $\Sigma^1_1$. The \LL-formula $\neg\phi$ belongs to  $\Sigma^1_1$ if and only if $\phi$ is equivalent to a first-order formula. 
\end{description}
\end{theorem}
\begin{proof}
(i) It suffices to prove the direction ``$\Longrightarrow$''. Suppose  both $\phi$ and $\neg\phi\vee\forall \vec{x}\neg R\vec{x}$ belong to $\Sigma^1_1$. We may assume without loss of generality that $\phi\equiv \exists S_1\dots\exists S_k\psi$ and $(\neg\phi\vee\forall \vec{x}\neg R\vec{x})\equiv\exists T_1\dots\exists T_m\chi$ for some first-order formulas $\psi$ and $\chi$, and the relation variables $S_1,\dots,S_k,T_1,\dots,T_m$ are fresh and pairwise distinct. Assume also that $\phi(R)$ and $\mathop{\exists S_1\dots\exists S_k\psi}$ are $\LL_1(R)$-formulas, and $\neg\phi(R)\vee\forall \vec{x}\neg R\vec{x}$ and $\exists T_1\dots\exists T_m\chi$ are $\LL_2(R)$-formulas.

\begin{clm}[1] 
$\psi\models\neg\chi\vee\forall \vec{x}\neg R\vec{x}$.
\end{clm}
\begin{proofclaim}[1]
Put $\mathcal{L}=\mathcal{L}_1\cup\mathcal{L}_2\cup\{R,S_1,\dots,S_k,T_1,\dots,T_m\}$. For any $\mathcal{L}$-model $M$ such that $M\models \psi$, we have $M\models\exists S_1\dots\exists S_k\psi$, thereby $M\models\phi$. If $R^M=\emptyset$, then $M\models\forall \vec{x}\neg R\vec{x}$, thereby $M\models\neg\chi\vee\forall \vec{x}\neg R\vec{x}$. If $R^M\neq\emptyset$, then we have $M\models\neg\forall \vec{x}\neg R\vec{x}$. Next, it follows that
\begin{align*}
M\models\neg\neg\phi\wedge \neg\forall \vec{x}\neg R\vec{x}&\Longrightarrow M\models\neg(\neg\phi\vee \forall \vec{x}\neg R\vec{x})
\Longrightarrow M\models\neg \exists T_1\dots\exists T_m\chi\\
&\Longrightarrow M\models\forall T_1\dots\forall T_m\neg\chi
\Longrightarrow M\models\neg\chi\\
&\Longrightarrow M\models\neg\chi\vee\forall \vec{x}\neg R\vec{x}
\end{align*}
as required.
\end{proofclaim}

Now, by Craig's Interpolation Theorem of first-order logic, there exists a first-order $\LL_1(R)\cap\LL_2(R)$-formula $\theta$ such that $\psi\models\theta$ and $\theta\models\neg\chi\vee\forall \vec{x}\neg R\vec{x}$.

\begin{clm}[2] 
$\phi\equiv\theta$.
\end{clm}
\begin{proofclaim}[2]
For any $\LL_1(R)$-model $M$, if $M\models\phi$, then $(M,S_1^M,\dots,S_k^M)\models\psi$ for some relations $S_1^M,\dots,S_k^M$ on $M$. Hence, $M\models\theta$.

Conversely, for any $\LL_1(R)$-model $M$ such that $M\not\models\phi$, we have $R^M\neq\emptyset$ and $M\models\neg\phi\vee \forall \vec{x}\neg R\vec{x}$. The latter implies  $(M,T_1^M,\dots,T_m^M)\models\chi$ for some relations $T_1^M,\dots,T_m^M$ on $M$. It then follows that $(M,T_1^M,\dots,T_m^M)\not\models\neg\chi\vee\forall \vec{x}\neg R\vec{x}$. Hence, $M\not\models\theta$.
\end{proofclaim}

(ii) The nontrivial direction ``$\Longrightarrow$'' follows from a similar and simplified argument. Instead of proving  Claim 1 as in (i), one proves $\psi\models\neg\chi$.
\end{proof}

Now, we are ready to give the proof of \Cref{neg_charac_thm_ind}.

\begin{proof}[Proof of \Cref{neg_charac_thm_ind}]
(i) Let $\phi$ be an \LL-formula  of \Ind. By \Cref{ind2sigma11}(ii) there exists an $\LL(R)$-sentence $\tau_\phi(R)$ of $\Sigma^1_1$ such that for all models $M$ and  teams $X$,
\begin{equation}\label{neg_char_thm_eq1}
M\models_X \cn\phi\iff M\not\models_X\phi\text{ or }X=\emptyset\iff (M,rel(X))\models \neg\tau_\phi(R)\vee\forall \vec{x}\neg R\vec{x}.
\end{equation}

Now, to prove the direction ``$\Longleftarrow$'', assume that $\tau_\phi(R)$ is equivalent to a first-order sentence. Then, the sentence $\neg\tau_\phi(R)$ is also equivalent to a first-order sentence, and thus by \Cref{ind2sigma11}(iv) there exists a formula $\rho(\neg\tau_\phi)$ of \Ind such that for all \LL-models $M$ and  teams $X$,
\[M\models_X\rho(\neg\tau_\phi)\iff (M,rel(X))\models\neg\tau_\phi(R)\vee\forall \vec{x}\neg R\vec{x}.\]
It then follows from (\ref{neg_char_thm_eq1}) that $\rho(\neg\tau_\phi)\equiv \cn\phi$.

Finally, to prove the direction ``$\Longrightarrow$'', assume that $\cn\phi\equiv \psi$ for some formula $\psi$ of \Ind. By \Cref{ind2sigma11}(ii) there exists an $\LL(R)$-sentence $\tau_\psi(R)$ of $\Sigma^1_1$ such that for all  models $M$ and  teams $X$,
\[M\models_X\psi\iff (M,rel(X))\models\tau_\psi(R).\]
By (\ref{neg_char_thm_eq1}) and the correspondence between teams and relations, $\tau_\psi(R)\equiv \neg\tau_\phi(R)\vee\forall \vec{x}\neg R\vec{x}$ and thereby the formula $\neg\tau_\phi(R)\vee\forall \vec{x}\neg R\vec{x}$ belongs to $\Sigma^1_1$. For any model $M$, since $M\models_\emptyset\phi$, we have $(M,\emptyset)\models\tau_\phi(R)$, where $rel(\emptyset)=\emptyset$. Then, by \Cref{neg_char_thm_so}(i), we conclude that $\tau_\phi(R)$ is equivalent to a first-order formula.

(ii) This item is proved by a similar argument that makes use of \Cref{ind2sigma11}(i) and \Cref{neg_char_thm_so}(ii). 
\end{proof}

\section{Axiomatizing negatable consequences in dependence and independence logic}\label{sec:axiom_fo_cns}

Dependence and independence logic are not (effectively) axiomatizable, meaning that  the consequence relation $\Gamma\models\phi$ cannot be effectively axiomatized.
Nevertheless, if we restrict $\Gamma$ to a set of sentences and $\phi$ to a first-order sentence, the consequence relation $\Gamma\models\phi$ is (effectively) axiomatizable and explicit axiomatizations for \D and \Ind are given in \cite{Axiom_fo_d_KV} and \cite{Hannula_fo_ind_13}. Throughout this section, let $\mathsf{L}$ denote one of the logics of \D and \Ind, and $\vdash_{\mathsf{L}}$ denote the syntactic consequence relation associated with the deduction system of $\mathsf{L}$ defined in \cite{Axiom_fo_d_KV} or in \cite{Hannula_fo_ind_13}.

\begin{theorem}[see \cite{Axiom_fo_d_KV,Hannula_fo_ind_13}]\label{ind_fo_cons_axiom}
If $\Gamma$ is a set of sentences  of $\mathsf{L}$, and $\phi$ is a first-order sentence,  then $\Gamma\models\phi\iff \Gamma\vdash_{\mathsf{L}}\phi$. In particular, $\Gamma\models\bot\iff \Gamma\vdash_{\mathsf{L}}\bot$.
\end{theorem}

Kontinen \cite{Kontinen15foc} generalized the above axiomatization result to cover also the case when $\Gamma\cup\{\phi\}$ is a set of formulas (that possibly contain free variables) by adding a new relation symbol to interpret the teams associated with the free variables. In this section, we will generalize \Cref{ind_fo_cons_axiom} without expanding the signature to cover the case when $\Gamma\cup\{\phi\}$ is a set of formulas (that possibly contain free variables) and $\phi$ is negatable.

First, note that by applying the following (standard) existential quantifier introduction and elimination rule (included in the systems given in  \cite{Axiom_fo_d_KV} and \cite{Hannula_fo_ind_13}):


\begin{center}
\renewcommand{\arraystretch}{1.8}
\begin{tabular}{|C{0.46\linewidth}|C{0.46\linewidth}|}
\hline
\AxiomC{$\phi(t/x)$}\RightLabel{$\exists$\textsf{I}}\UnaryInfC{$\exists x\phi$}\DisplayProof
&\AxiomC{$D_1$}\noLine\UnaryInfC{$\exists x\phi$} \AxiomC{}\noLine\UnaryInfC{$[\phi]$}\noLine\UnaryInfC{$D_2$}\noLine\UnaryInfC{$\psi$} \RightLabel{$\exists$\textsf{E}}\BinaryInfC{$\psi$}\DisplayProof\\
&\footnotesize   $x$ does not occur freely in $\psi$ or  any formula in the undischarged assumptions of the derivation $D_2$\\\hline
\end{tabular}
\end{center}


\noindent under certain constraint the (possibly open) formula $\psi$ in the entailment $\Delta,\psi\vdash_{\mathsf{L}}\theta$ can be turned into a sentence without affecting the entailment relation, as  the lemma below shows.

\begin{lem}\label{form2sent_lm2}
Let $\Delta\cup\{\chi,\theta\}$ be a set of formulas  of $\mathsf{L}$. If the free variables $x_1,\dots,x_n$ of $\chi$ do not occur freely in  $\theta$ or any formula in $\Delta$, then $\Delta,\chi\vdash_{\mathsf{L}}\theta\iff \Delta,\exists x_1\dots\exists x_n\chi\vdash_{\mathsf{L}}\theta$.
\end{lem}
\begin{proof}
Apply the rules $\exists$\textsf{I} and $\exists$\textsf{E}.
\end{proof}

%
%
%
%

To understand intuitively why \Cref{ind_fo_cons_axiom} can be generalized, let us 
consider a set $\Gamma\cup\{\phi\}$ of formulas of $\mathsf{L}$. Since $\mathsf{L}$ is equivalent to $\Sigma^1_1$ by \Cref{ind2sigma11}, and $\Sigma^1_1$ admits the Compactness Theorem, we may assume that $\Gamma$ is a finite set. We shall add to the deduction system of $\mathsf{L}$ the (sound) rule to guarantee that $\Gamma\vdash\phi$ follows from $\Gamma,\cn\phi\vdash\bot$. The latter, by \Cref{form2sent_lm2}, is equivalent to  $\exists \vec{x}(\bigwedge\Gamma\wedge \cn\phi)\vdash\bot$, where $\vec{x}$ lists all free variables in $\Gamma$ and $\phi$. 
Then, the Completeness Theorem can be restated as 
\(\exists \vec{x}(\bigwedge\Gamma\wedge \cn\phi)\not\vdash\bot\Longrightarrow\exists \vec{x}(\bigwedge\Gamma\wedge \cn\phi)\not\models\bot.\)
Assuming that $\exists \vec{x}(\bigwedge\Gamma\wedge \cn\phi)$ is deductively consistent, the problem reduces to the problem of constructing a model for the sentence $\exists \vec{x}(\bigwedge\Gamma\wedge \cn\phi)$, which further reduces to the problem of constructing a model for the $\Sigma^1_1$ sentence $\tau_{\exists \vec{x}(\bigwedge\Gamma\wedge \cn\phi)}$, in case $\cn\phi$ is definable in $\mathsf{L}$.
 Finding a model for $\tau_{\exists \vec{x}(\bigwedge\Gamma\wedge \cn\phi)}$, which we may assume to have the form $\exists\vec{R}\theta$, is the same as finding a model for the first-order formula $\theta$. 
This argument shows that via the trick of weak classical negation, \Cref{ind_fo_cons_axiom} can, in principle, be generalized. Note that if $\Gamma$ is a set of sentences and $\phi$ is a first-order sentence, then $\neg\phi\equiv\cn\phi$ and the foregoing argument reduces to the argument given in \cite{Axiom_fo_d_KV}. 

Let us now make this idea precise. Given the Completeness Theorems in \cite{Axiom_fo_d_KV} and \cite{Hannula_fo_ind_13}, it suffices to extend the systems natural deduction of \cite{Axiom_fo_d_KV} and \cite{Hannula_fo_ind_13} by adding the following (usual) elimination rule for the weak classical negation below to ensure that  $\Gamma\vdash\phi$ follows from  $\Gamma,\cn\phi\vdash\bot$ in case $\phi$ is negatable, where we use $\cn\phi$ as a shorthand for  the (up to equivalence unique) formula $\psi$ in the language of $\mathsf{L}$ such that $\psi\equiv\cn\phi$:
%
%
%
\vspace{-0.5\baselineskip}
\begin{center}
\renewcommand{\arraystretch}{1.8}
\begin{tabular}{|C{0.96\linewidth}|}
\multicolumn{1}{c}{\textbf{NEW RULE}}\\\hline
 \AxiomC{}\noLine\UnaryInfC{$[\cn\phi]$}\noLine\UnaryInfC{$\vdots$}\noLine\UnaryInfC{$\bot$} \RightLabel{\cn\!\textsf{E}}\UnaryInfC{$\phi$}\noLine\UnaryInfC{} \DisplayProof\\\hline
\end{tabular}
\end{center}



Let $\vdash_{\mathsf{L}}^\ast$ denote the syntactic consequence relation associated with the system of $\mathsf{L}$ extended with the rule  \cn\textsf{E}. We now prove the Soundness and Completeness Theorem for this extended system.


\begin{theorem}\label{general_axiom_compl}
If $\Gamma\cup\{\phi\}$ is a set of formulas of $\mathsf{L}$ such that $\phi$ is negatable, then we have $\Gamma\models\phi\iff \Gamma\vdash_{\mathsf{L}}^\ast\phi$.
\end{theorem}
\begin{proof}
``$\Longleftarrow$'': The soundness of the rules of the systems of $\mathsf{L}$ follows from \cite{Axiom_fo_d_KV} and \cite{Hannula_fo_ind_13}, and the new rule  \cn\!\textsf{E} is clearly sound.



``$\Longrightarrow$'':  
Since $\mathsf{L}$ is compact, we may w.l.o.g. assume $\Gamma$ to be finite. Then, we have
\begin{align*}
\Gamma\models\phi&\Longrightarrow\Gamma,\cn\phi\models\bot\Longrightarrow\exists \bar{x}(\bigwedge\Gamma\wedge\cn \phi)\models\bot\\
&\Longrightarrow \exists \bar{x}(\bigwedge\Gamma\wedge\cn \phi)\vdash_{\mathsf{L}} \bot\tag{by the Completeness Theorem of $\mathsf{L}$, \Cref{ind_fo_cons_axiom}}\\
&\Longrightarrow \Gamma,\cn \phi\vdash_{\mathsf{L}} \bot\tag{by \Cref{form2sent_lm2}}\\
&\Longrightarrow \Gamma\vdash_{\mathsf{L}}^\ast\phi,\tag{by  \cn\!\textsf{E}}
\end{align*}
where $\cn\phi$ is a shorthand for the formula $\psi$ in the language of $\textsf{L}$ such that $\psi\equiv\cn\phi$.
\end{proof}



If $\phi(\vec{x})$ is first-order formula, then  its $\Sigma^1_1$-translation is equivalent to the first-order sentence $\tau_\phi(R)=\forall \vec{x}(R(\vec{x})\to \phi(\vec{x}))$. We thus know by \Cref{neg_charac_thm_ind} that $\phi$ is negatable in \Ind, and also in \D in case $\phi$ is a sentence. This shows that our  \Cref{general_axiom_compl} is indeed a generalization of \Cref{ind_fo_cons_axiom} and also of \cite{Kontinen15foc} for \Ind.

Let us stress here that the new rule $\cn\textsf{E}$ in the extended system is  to be applied only to negatable formulas $\phi$, and the defining formula of $\cn\phi$ in the original language of the logic $\mathsf{L}$ has to be computed before applying the rule. Since the class of negatable formulas is undecidable (as discussed in Section 3) in the first place, the rule $\cn\textsf{E}$ is  not entirely effective. Nevertheless, in the next section we will identify an interesting (decidable) class of negatable formulas in \Ind whose weak classical negations have uniform translations in \Ind. For this and possibly other interesting classes of negatable formulas the rule $\cn\textsf{E}$ is applicable effectively.


We now end this section by providing a uniform translation in \Ind for the weak classical negation $\cn\phi$ of every first-order formula $\phi$. The translation we give in the proposition below is much more efficient and succinct than the translation obtained by going back and forth through the lengthy $\Sigma^1_1$-translation of the formula (i.e. by applying \Cref{ind2sigma11}(ii)(iv), see \cite{Van07dl,Pietro_I/E} for details).



\begin{prop}\label{neg_dual_lm}
If $\phi$ is a first-order formula, then $\cn\phi(\vec{x})\equiv\exists \vec{w}(\vec{w}\subseteq \vec{x}\wedge \neg\phi(\vec{w}))$, where $\vec{w}$ is a sequence of fresh variables. 
\end{prop} 
\begin{proof}
For any  model $M$ and team $X$, since $\phi$ is flat, we have
\[
M\models_X\cn\phi\iff X=\emptyset\text{ or }M\not\models_X\phi\iff X=\emptyset\text{ or }\exists s\in X(M\not\models_{\{s\}}\phi(\vec{x})).
\]
In view of the empty team property of \Ind , it  suffices to show that for any nonempty team $X$,
\[ \exists s\in X(M\not\models_{\{s\}}\phi(\vec{x}))\iff M\models_X\exists \vec{w}(\vec{w}\subseteq \vec{x}\wedge \neg\phi(\vec{w})).\]

%


``$\Longrightarrow$": Let $\vec{x}=\langle x_1,\dots, x_n\rangle$. Assume $M\not\models_{\{s\}}\phi(\vec{x})$ for some $s\in X$. Define inductively  a constant function $F_i$ for each $1\leq i\leq n$
as follows:
\begin{itemize}
\item $F_1:X\to \wp(M)\setminus\{\emptyset\}$ is defined as $F_1(t)=\{s(x_1)\}$;
\item $F_i:X(F_1/w_1,\dots,F_{i-1}/w_{i-1})\to  \wp(M)\setminus\{\emptyset\}$ is defined as $F_i(t)=\{s(x_i)\}$.
\end{itemize}
Clearly, $X(\vec{F}/\vec{w})=X(s(\vec{x})/\vec{w})$, and thus  $M\models_{X(\vec{F}/\vec{w})}\vec{w}\subseteq \vec{x}$. On the other hand, for any $t\in X(\vec{F}/\vec{w})$, since $t(\vec{w})=s(\vec{x})$ and $M\not\models_{\{s\}}\phi(\vec{x})$, we obtain $M\not\models_{\{t\}}\phi(\vec{w})$ by the locality property, which means $M\models_{\{t\}}\neg\phi(\vec{w})$ by the law of excluded middle for first-order formulas. Hence, we obtain $M\models_{X(\vec{F}/\vec{w})}\neg\phi(\vec{w})$ by the flatness property.

``$\Longleftarrow$'': Conversely, suppose $M\models_X\exists \vec{w}(\vec{w}\subseteq \vec{x}\wedge \neg\phi(\vec{w}))$. Then there exist appropriate functions $F_i$ ($1\leq i\leq n$) for the existential quantifications $\exists\vec{w}$ such that $M\models_{X(\vec{F}/\vec{w})}\vec{w}\subseteq \vec{x}$ and $M\models_{X(\vec{F}/\vec{w})}\neg\phi(\vec{w})$. The latter implies that $M\not\models_{\{t\}}\phi(\vec{w})$ for an arbitrary $t\in X(\vec{F}/\vec{w})$. On the other hand, by the former, there exists $s'\in X(\vec{F}/\vec{w})$ such that $s'(\vec{x})=t(\vec{w})$. For $s=s'\upharpoonright {\rm dom}(X)\in X$, we have $s(\vec{x})=s'(\vec{x})=t(\vec{w})$, which implies $M\not\models_{\{s\}}\phi(\vec{x})$ by the locality property.
\end{proof}

\section{A hierarchy of negatable atoms}

In this section, we define an interesting class of formulas that are negatable in \Ind. This class will be presented in the form of an alternating hierarchy of atoms that are definable in \Ind, and the weak classical negation of each such atom also falls in this hierarchy. These atoms are closely related to  the \emph{dependency notions} considered in \cite{Galliani2014SFOD}, and the \emph{generalized dependence atoms} studied in \cite{Kuusisto2015} and \cite{KontinenMullerSchnoorVollmer2014}. We will demonstrate that all first-order formulas, dependence atoms, independence atoms and inclusion atoms belong to this class. At the end of the section, we will also show that the set of negatable formulas is closed under Boolean connectives and the weak quantifiers. 
For all this type of negatable formulas, the Completeness Theorem of \Ind we obtained in the previous section applies, and in this section we will also give  uniform translations for the weak classical negations of these negatable formulas in \Ind.


Let us start by defining the notion of abstract  relation. A $k$-ary \emph{abstract relation} $R$  is a class of pairs $(M,R^M)$ that is closed under isomorphisms, where $M$ ranges over first-order models and $R^M\subseteq M^k$. For instance, the familiar equality $=$ is a binary abstract relation defined by the class
\[\{ (M, =^M)\mid M\text{ is a first-order model} \},\text{ where }=^M:=\{(a,a)\mid a\in M\}.\]
Every first-order formula $\phi(x_1,\dots,x_k)$ with $k$ free variables induces a $k$-ary abstract relation 
\[\pmb{\phi}:=\{(M,\pmb{\phi}^M)\mid M\text{ is a first-order model}\},\]
where $\pmb{\phi}^M:=\{(s(x_1),\dots,s(x_k))\mid M\models_s\phi\}$.
A $k$-ary abstract relation $R$ is said to be \emph{(first-order) definable} if there exists a (first-order) formula $\phi_R(w_1,\dots,w_k)$ 
such that for all models $M$ and  assignments $s$,
\[s(\vec{w})\in R^M \iff M\models_{s}\phi_{R}(\vec{w}).\] 
Clearly, the first-order formula $w=u$ defines the abstract equality relation, and every first-order formula $\phi$ defines its associated abstract relation $\pmb{\phi}$. 

If $R$ is a $k$-ary abstract relation, then we write $\overline{R}$ for the complement of $R$ that is defined by letting $\overline{R}^M=M^k\setminus R^M$ for all models $M$. Clearly, if a first-order formula $\phi$ defines $R$, then its negation $\neg\phi$ defines $\overline{R}$.

If $\vec{s}=\langle s_1,\dots,s_k\rangle$, then we write $\vec{s}(\vec{x})$ for $\langle s_1(\vec{x}),\dots,s_k(\vec{x})\rangle$. For every sequence $\mathsf{k}=\langle k_1,\dots,k_n\rangle$ of natural numbers and every  $(k_1+\dots+k_n)\cdot m$-ary abstract relation $R$, we introduce two new atomic formulas $\Sigma_{n,\mathsf{k}}^R(x_1,\dots,x_m)$ and $\Pi_{n,\mathsf{k}}^R(x_1,\dots,x_m)$  with the semantics  defined as follows:
\begin{itemize}
\item $M\models_\emptyset\Sigma_{n,\mathsf{k}}^R(\vec{x})$ and $M\models_\emptyset\Pi_{n,\mathsf{k}}^R(\vec{x})$.

\item If $n$ is odd, then  define for any model $M$ and any \emph{nonempty} team $X$,
\begin{itemize}
\item $M\models_X\Sigma_{n,\mathsf{k}}^R(\vec{x})$ ~iff~ there exist $s_{11},\dots,s_{1k_1}\in X$ such that  for all $s_{21},\dots,s_{2k_2}\in X$, there exist $\dots$  there exist $s_{n1},\dots,s_{nk_n}\in X$ such that $(\vec{s_1}(\vec{x}),\dots,\vec{s_{n}}(\vec{x}))\in R^M$;
\item $M\models_X\Pi_{n,\mathsf{k}}^R(\vec{x})$ ~iff~ for all $s_{11},\dots,s_{1k_1}\in X$, there exist $s_{21},\dots,s_{2k_2}\in X$ such that  for all $\dots$  for all $s_{n1},\dots,s_{nk_n}\in X$, it holds that $(\vec{s_1}(\vec{x}),\dots,\vec{s_{n}}(\vec{x}))\in R^M$.
\end{itemize}


\item Similarly if $n$ is even.
\end{itemize}

\begin{fact}\label{neg_sig_neg_pi}
$\cn\Sigma_{n,\mathsf{k}}^R(\vec{x})\equiv \Pi_{n,\mathsf{k}}^{\overline{R}}(\vec{x})$ and $\cn\Pi_{n,\mathsf{k}}^R(\vec{x})\equiv \Sigma_{n,\mathsf{k}}^{\overline{R}}(\vec{x})$.
\end{fact}

Let us now give some examples of  the  $\Sigma_{n,\mathsf{k}}^R$ and $\Pi_{n,\mathsf{k}}^R$ atoms.

\begin{exmp}\label{neg_atom_example}
\begin{description}
\item[(a)] The dependence atom $\dep(x_1,\dots,x_k,y)$ is the $\Pi^{\mathbf{dep}_k}_{1,\langle 2\rangle}(x_1,\dots,x_k,y)$ atom, where $\mathbf{dep}_k$ is the $2(k+1)$-ary abstract relation defined as
\[(a_1,\dots,a_k,b,a_1',\dots,a_k',b')\in( \mathbf{dep}_k)^M ~\text{ iff }~a_1,\dots,a_k=a_1',\dots,a_k'\Longrightarrow b=b'.\]
Clearly,  $\mathbf{dep}_k$ is definable by the  first-order formula 
\[\phi_{\mathbf{dep}_k}=\big(w_1=w_1'\wedge \dots\wedge w_k=w_k'\big)\to u=u'.\]
\item[(b)] The independence atom $x_1,\dots,x_k\perp_{z_1,\dots,z_n} y_1,\dots,y_m$ is the \vspace{-0.5\baselineskip}
\[\Pi^{\mathbf{ind}_{kmn}}_{2,\langle 2,1\rangle}(x_1,\dots,x_k,y_1,\dots,y_m,z_1,\dots,z_n)\] 
atom, where $\mathbf{ind}_{kmn}$ is the (first-order definable) $(2+1)(k+m+n)$-ary abstract  relation defined as 
\((\vec{a},\vec{b},\vec{c},\vec{a'},\vec{b'},\vec{c'},\vec{a''},\vec{b''},\vec{c''})\in (\mathbf{ind}_{kmn})^M\) iff\vspace{-0.25\baselineskip}
\begin{align*}
(c_n,\dots,c_n)=(c_1',&\dots,c_n')\Longrightarrow[\,c_1,\dots,c_n=c_1'',\dots,c_n'',\\
&  a''_1,\dots,a''_k=a_1,\dots,a_k\text{ and }b''_1,\dots,b''_m=b_1',\dots,b_m'\,].
\end{align*}
\item[(c)] The inclusion atom $x_1,\dots,x_k\subseteq y_1,\dots,y_k$ is the $\Pi^{\mathbf{inc}_k}_{2,\langle 1,1\rangle}(x_1,\dots,x_k,y_1,\dots,y_k)$ atom, where $\mathbf{inc}_k$ is the (first-order definable) $(1+1)2k$-ary abstract relation defined as\vspace{-0.25\baselineskip}
\[(a_1,\dots,a_k,b_1,\dots,b_k,a_1',\dots,a_k',b_1',\dots,b_k')\in (\mathbf{inc}_k)^M~\text{ iff }~a_1,\dots,a_k=b_1',\dots,b_k'.\]
\item[(d)] Every first-order formula $\phi(x_1,\dots,x_k)$ is the $\Pi^{\pmb{\phi}}_{1,\langle 1\rangle}(x_1,\dots,x_k)$ atom, where $\pmb{\phi}$ is the (first-order definable) $1\cdot k$-ary abstract relation defined as\vspace{-0.25\baselineskip}
\[(a_1,\dots,a_k)\in \pmb{\phi}^M~\text{ iff }~M\models_{s_{\vec{a}}}\phi\text{ where }s_{\vec{a}}(x_i)=a_i\text{ for all }i.\]
\end{description}
\end{exmp}

In what follows, let $\mathsf{k}=\langle k_1,\dots,k_n\rangle$ be an arbitrary sequence   of natural numbers, $\vec{x}=\langle x_1,\dots,x_m\rangle$  an arbitrary sequence of variables, and $R$ an arbitrary $(k_1+\dots+k_n)m$-ary abstract relation. 
Suppose $R$ is definable by a formula $\phi_R(\vv{\mathsf{w}_1},\dots,\vv{\mathsf{w}_n})$, where $\vv{\mathsf{w}_i}=\langle \mathsf{w}_{i,1},\dots,\mathsf{w}_{i,k_i}\rangle$ and $\mathsf{w}_{i,j}=\langle w_{i,j,1},\dots,w_{i,j,m} \rangle$. 
The $\Sigma_{n,\mathsf{k}}^R(\vec{x})$ and $\Pi_{n,\mathsf{k}}^R(\vec{x})$ atoms can be translated into second-order logic in the same manner as in \Cref{ind2sigma11}. For instance, if $n$ is odd,  let $S$ be a fresh $m$-ary relation symbol and let $\tau_{\Sigma^R_{n,\mathsf{k}}}(S):=$\vspace{-0.5\baselineskip}
\[
\begin{array}{rl}
\exists \vv{\mathsf{w}_1}\Big(S(\mathsf{w}_{1,1})\wedge\dots\wedge &\!\!\!\!\!\!S(\mathsf{w}_{1,k_1})\wedge \forall \vv{\mathsf{w}_2}\Big(S(\mathsf{w}_{2,1})\wedge\dots\wedge S(\mathsf{w}_{2,k_2})\to\exists \vv{\mathsf{w}_3}\cdots\\
&\cdots\exists \vv{\mathsf{w}_n}\Big(S(\mathsf{w}_{n,1})\wedge\dots\wedge S(\mathsf{w}_{n,k_n})\wedge\phi_R(\vv{\mathsf{w}_1},\dots,\vv{\mathsf{w}_n})\underbrace{\Big)\cdots\Big)\Big)}_{n}.
\end{array}
\]
It is easy to verify that  $M\models_X\Sigma_{n,\mathsf{k}}^R(\vec{x})\iff (M,rel(X))\models\tau_{\Sigma_{n,\mathsf{k}}^R}(S)$ for any model $M$ and team $X$. If $\phi_R(\vv{\mathsf{w}_1},\dots,\vv{\mathsf{w}_n})$ is a first-order formula, i.e., if $R$ is first-order definable, then $\tau_{\Sigma_{n,\mathsf{k}}^R}(S)$ is a first-order sentence. Therefore, by \Cref{neg_charac_thm_ind}(i), the $\Sigma_{n,\mathsf{k}}^R(\vec{x})$ and $\Pi_{n,\mathsf{k}}^R(\vec{x})$ atoms with first-order definable $R$ are negatable in \Ind. 

In order to apply the rules of the extended deduction system defined in Section 4 to derive the $\Sigma_{n,\mathsf{k}}^R(\vec{x})$ and $\Pi_{n,\mathsf{k}}^R(\vec{x})$ consequences in \Ind, one needs to 
compute the formulas that are equivalent to the weak classical negations of the $\Sigma_{n,\mathsf{k}}^R(\vec{x})$ and $\Pi_{n,\mathsf{k}}^R(\vec{x})$ atoms in the original language of \Ind.
This can be done by applying \Cref{neg_sig_neg_pi} and going back and forth through  the  $\Sigma^1_1$-translation, which is however inefficient. In what follows, we will give a direct  definition of the atoms $\Sigma_{n,\mathsf{k}}^R(\vec{x})$ and $\Pi_{n,\mathsf{k}}^R(\vec{x})$ in the original language of \Ind.



For each $1\leq i\leq n$, recall $\vv{\mathsf{w}_i}=\langle\mathsf{w}_{i,1},\dots, \mathsf{w}_{i,k_i}\rangle$, and define
\begin{itemize}
\item $\displaystyle\mathsf{inc}(\vv{\mathsf{w}_{i}};\vec{x})\,:=\,\bigwedge_{j=1}^{k_i}\mathsf{w}_{i,j}\subseteq \vec{x}$,
\item $\displaystyle\mathsf{pro}(\vv{\mathsf{w}_1},\dots,\vv{\mathsf{w}_{i-1}};\vec{x};\vv{\mathsf{w}_{i}})\,:=$

\hfill$\displaystyle \vv{\mathsf{w}_1}\dots\vv{\mathsf{w}_{i-1}}\perp\vv{\mathsf{w}_{i}}\,\wedge\bigwedge_{j=1}^{k_i}\big(\vec{x}\subseteq \mathsf{w}_{i,j}\wedge\, \langle \mathsf{w}_{i,j'}\mid j'\neq j\rangle\!\perp\mathsf{w}_{i,j}\big)$\footnote{If $i=1$, then $\vv{\mathsf{w}_1}\dots\vv{\mathsf{w}_{i-1}}$ denotes the empty sequence $\langle\rangle$, and we stipulate $\langle\rangle\!\perp\vec{y}:=\top$ for any  $\vec{y}$.}.
\end{itemize}
Define then inductively formulas $\sigma_i$ and $\pi_i$ as follows:

\begin{itemize}
   \setlength{\itemsep}{2pt}
    \setlength{\parskip}{2pt}
    \setlength{\parsep}{1pt}
\item $\sigma_{1}[\vec{x},\phi_R(\vv{\mathsf{w}_1},\dots,\vv{\mathsf{w}_n})]:=~ \exists\vv{\mathsf{w}_n}\Big(  \mathsf{inc}(\vv{\mathsf{w}_{n}};\vec{x})\wedge\phi_R(\vv{\mathsf{w}_1},\dots,\vv{\mathsf{w}_n})\Big)$,
\item $\pi_{1}[\vec{x},\phi_R(\vv{\mathsf{w}_1},\dots,\vv{\mathsf{w}_n})]:=\exists \vv{\mathsf{w}_n}\Big(\mathsf{pro}(\vv{\mathsf{w}_1},\dots,\vv{\mathsf{w}_{n-1}};\vec{x};\vv{\mathsf{w}_{n}})\wedge\phi_R(\vv{\mathsf{w}_1},\dots,\vv{\mathsf{w}_n})\Big)$,
\item $\sigma_{i+1}[\vec{x},\phi_R(\vv{\mathsf{w}_1},\dots,\vv{\mathsf{w}_n})]:=\exists\vv{\mathsf{w}_{n-i}}\Big(  \mathsf{inc}(\vv{\mathsf{w}_{n-i}};\vec{x})\wedge \pi_{i}[\vec{x},\phi_R(\vv{\mathsf{w}_1},\dots,\vv{\mathsf{w}_n})]\Big)$,
\item $\pi_{i+1}[\vec{x},\phi_R(\vv{\mathsf{w}_1},\dots,\vv{\mathsf{w}_n})]\displaystyle:=\exists \vv{\mathsf{w}_{n-i}}\Big(\mathsf{pro}(\vv{\mathsf{w}_1},\dots,\vv{\mathsf{w}_{n-i-1}};\vec{x};\vv{\mathsf{w}_{n-i}})$

\hfill$ \wedge\, \sigma_{i}[\vec{x},\phi_R(\vv{\mathsf{w}_1},\dots,\vv{\mathsf{w}_n})]\Big)$.

\end{itemize}

\begin{theorem}\label{sigma_pi_df_ind}
Let $R$ and $\phi_R$ be as above. Then 
\begin{itemize}
\item $\Sigma_{n,\mathsf{k}}^R(\vec{x})\equiv \sigma_n[\vec{x},\phi_R(\vv{\mathsf{w}_1},\dots,\vv{\mathsf{w}_n})]$,
\item $\Pi_{n,\mathsf{k}}^R(\vec{x})\equiv \pi_n[\vec{x},\phi_R(\vv{\mathsf{w}_1},\dots,\vv{\mathsf{w}_n})]$.
\end{itemize}
\end{theorem}

Before proving the above theorem, let us consider some examples.  As discussed in \Cref{neg_atom_example}(a), the dependence atom $\dep(x_1,\dots,x_k,y)$ is a $\Pi^{\mathbf{dep}_k}_{1,\langle 2\rangle}(\vec{x}y)$. By \Cref{neg_sig_neg_pi}, the weak classical negation $\cn\dep(\vec{x},y)$ of the dependence atom is the atom $\Sigma^{\overline{\mathbf{dep}_k}}_{1,\langle 2\rangle}(\vec{x}y)$, which, according to the above theorem, is expressible in \Ind by the  formula
\begin{equation}\label{neg_dep_frm}
\sigma_1[\vec{x}y,\phi_{\overline{\mathbf{dep}_k}}]=\exists \vec{w}_1v_1\vec{w}_2v_2(\vec{w}_1v_1\subseteq \vec{x}y\,\wedge\,\vec{w}_2v_2\subseteq \vec{x}y\,\wedge\, \vec{w}_1=\vec{w}_2\,\wedge\, v_1\neq v_2).
\end{equation}
The weak classical negation of independence atoms and inclusion atoms  (which are $\Sigma_{2,\mathsf{k}}^R$ atoms) can be computed from the above theorem in the same way. There are also more succinct definitions for these negated atoms than the ones obtained by applying the above theorem:
\begin{equation}\label{neg_ind_smpdf}
\cn\vec{x}\perp_{\vec{z}}\vec{y}\equiv \exists^1\vec{u}\vec{v}\vec{w}(\vec{u}\vec{w}\subseteq \vec{x}\vec{z}\,\wedge\, \vec{v}\vec{w}\subseteq \vec{y}\vec{z}\,\wedge\, \vec{u}\vec{v}\vec{w}\neq\vec{x}\vec{y}\vec{z}),
\end{equation}
\[\cn \vec{x}\subseteq\vec{y}\equiv\exists^1 \vec{z}(\vec{z}\subseteq \vec{x}\,\wedge\, \vec{z}\neq\vec{y}),\]
where the quantifier $\exists^1$ is defined as $\exists^1x\phi:=\exists x(\dep(x)\wedge\phi)$ (and we will come back to this quantifier later in the section). 
Finding simpler translations for other $\Sigma_{n,\mathsf{k}}^R$ and $\Pi_{n,\mathsf{k}}^R$ atoms with small $n$ is left as future work.

Now, to prove \Cref{sigma_pi_df_ind}, let us first prove a lemma concerning the crucial subformulas $\mathsf{inc}$ and $\mathsf{pro}$ in the $\sigma_n$ and $\pi_n$ formulas. In the following, for any team $Y$ of a model $M$, we write  $Y[\vec{x}]$ for the image of $Y$ on $\vec{x}$, i.e., $Y[\vec{x}]=\{s(\vec{x})\mid s\in Y\}$. For any set $A\subseteq M^n$ of tuples in $M$ and sequence $\mathsf{w}=\langle w_1,\dots,w_n\rangle$ of variables, define $Y(A/\mathsf{w})=\{s(a_1/w_1)\dots(a_n/w_n)\mid s\in X,~\langle a_1,\dots,a_n\rangle\in A\}$. We  abbreviate $X(A/\mathsf{w}_1)\cdots(A/\mathsf{w}_k)$ as $X(A/\mathsf{w}_1,\cdots,\mathsf{w}_k)$, and write simply $X(\mathsf{a}/\mathsf{w})$ for  $X(\{\mathsf{a}\}/\mathsf{w})$.



\begin{lem}\label{inc_pro_fact}
Let $1\leq i\leq n$, $M$ a model and $Y$ a nonempty team.
\begin{description}
\item[(i)] $M\models_Y\mathsf{inc}(\vv{\mathsf{w}_{i}};\vec{x})$ iff $Y[\mathsf{w}_{i,j}]\subseteq Y[\vec{x}]$ for all $1\leq j\leq k_i$. 
\item[(ii)]  If 
\begin{equation}\label{inc_pro_fact_eq00}
Y(Y[\vec{x}]/\mathsf{w}_{i,1},\dots,\mathsf{w}_{i,k_i})[\vv{\mathsf{w}_{1}},\dots,\vv{\mathsf{w}_{i}}]=Y[\vv{\mathsf{w}_{1}},\dots,\vv{\mathsf{w}_{i}}],
\end{equation}
then $M\models_Y\mathsf{pro}(\vv{\mathsf{w}_1},\dots,\vv{\mathsf{w}_{i-1}};\vec{x};\vv{\mathsf{w}_{i}})$.

For the converse direction, if $M\models_Y\mathsf{pro}(\vv{\mathsf{w}_1},\dots,\vv{\mathsf{w}_{i-1}};\vec{x};\vv{\mathsf{w}_{i}})$, then 
\begin{equation}\label{inc_pro_fact_eq0}
Y(Y[\vec{x}]/\vv{\mathsf{w}_{i}})[\vv{\mathsf{w}_{1}},\dots,\vv{\mathsf{w}_{i}}]\subseteq Y[\vv{\mathsf{w}_{1}},\dots,\vv{\mathsf{w}_{i}}].
\end{equation}
\end{description}
\end{lem}
\begin{proof}
Item (i) is obvious. For item (ii), the direction that (\ref{inc_pro_fact_eq00}) implies the clause $M\models_Y\mathsf{pro}(\vv{\mathsf{w}_1},\dots,\vv{\mathsf{w}_{i-1}};\vec{x};\vv{\mathsf{w}_{i}})$ is clear. We only give a detailed proof for the other direction.

Suppose that $M\models_Y\mathsf{pro}(\vv{\mathsf{w}_1},\dots,\vv{\mathsf{w}_{i-1}};\vec{x};\vv{\mathsf{w}_{i}})$. To prove (\ref{inc_pro_fact_eq0}), for an arbitrary element $\langle \vv{\mathsf{a}_{1}},\dots,\vv{\mathsf{a}_{i}}\rangle\in Y(Y[\vec{x}]/\mathsf{w}_{i,1},\dots,\mathsf{w}_{i,k_i})[\vv{\mathsf{w}_{1}},\dots,\vv{\mathsf{w}_{i}}]$, we show that $\langle \vv{\mathsf{a}_{1}},\dots,\vv{\mathsf{a}_{i}}\rangle\in Y[\vv{\mathsf{w}_{1}},\dots,\vv{\mathsf{w}_{i}}]$ as well. By definition,  for each $1\leq j\leq k_i$, there exists $s_j\in Y$ such that $s_j(\vec{x})=\mathsf{a}_{i,j}$. Moreover, since $M\models_Y\vec{x}\subseteq \mathsf{w}_{i,j}$, there exists $s_j'\in Y$ such that $s_j'(\mathsf{w}_{i,j})=s_j(\vec{x})=\mathsf{a}_{i,j}$. 

Now, since $M\models_Y\langle \mathsf{w}_{i,1},\mathsf{w}_{i,3},\dots,\mathsf{w}_{i,k_i}\rangle\perp \mathsf{w}_{i,2}$, we can find $s''_2\in Y$ such that $s''_2(\mathsf{w}_{i,1},\mathsf{w}_{i,2})=s_1'(\mathsf{w}_{i,1})s_2'(\mathsf{w}_{i,2})$. Also, since $M\models_Y\langle \mathsf{w}_{i,1},\mathsf{w}_{i,2},\mathsf{w}_{i,4},\dots,\mathsf{w}_{i,k_i}\rangle\perp \mathsf{w}_{i,3}$, there exists $s''_3\in Y$ such that 
\[s''_3(\mathsf{w}_{i,1},\mathsf{w}_{i,2},\mathsf{w}_{i,3})=s''_2(\mathsf{w}_{i,1},\mathsf{w}_{i,2})s'_3(\mathsf{w}_{i,3})=s_1'(\mathsf{w}_{i,1})s_2'(\mathsf{w}_{i,2})s_3'(\mathsf{w}_{i,3}).\] 
Proceeding in a similar way we find in the end an $s''_{k_i}\in Y$ such that 
\begin{equation*}
s''_{k_i}(\vv{\mathsf{w}_{i}})=s_1'(\mathsf{w}_{i,1})\dots s_{k_i}'(\mathsf{w}_{i,k_i})=s_1(\vec{x})\dots s_{k_i}(\vec{x})=\mathsf{a}_{i,1}\dots\mathsf{a}_{i,k_i}=\vv{\mathsf{a}_i}.
\end{equation*}

Now, the assumption $\langle \vv{\mathsf{a}_{1}},\dots,\vv{\mathsf{a}_{i}}\rangle\in Y(Y[\vec{x}]/\mathsf{w}_{i,1},\dots,\mathsf{w}_{i,k_i})[\vv{\mathsf{w}_{1}},\dots,\vv{\mathsf{w}_{i}}]$ also implies that there is $s\in Y$ such that $s(\vv{\mathsf{w}_{1}},\dots,\vv{\mathsf{w}_{i-1}})=\langle \vv{\mathsf{a}_{1}},\dots,\vv{\mathsf{a}_{i-1}}\rangle$. Since  $M\models_Y \vv{\mathsf{w}_1}\dots\vv{\mathsf{w}_{i-1}}\perp\vv{\mathsf{w}_{i}}$, it follows that there exists $s'\in Y$ such that 
\[s'(\vv{\mathsf{w}_{1}},\dots,\vv{\mathsf{w}_{i-1}},\vv{\mathsf{w}_{i}})=s(\vv{\mathsf{w}_{1}},\dots,\vv{\mathsf{w}_{i-1}})s''_{k_i}(\vv{\mathsf{w}_{i}})=\langle \vv{\mathsf{a}_{1}},\dots,\vv{\mathsf{a}_{i-1}},\vv{\mathsf{a}_{i}}\rangle.\]
From this we conclude that $\langle \vv{\mathsf{a}_{1}},\dots,\vv{\mathsf{a}_{i}}\rangle\in Y[\vv{\mathsf{w}_{1}},\dots,\vv{\mathsf{w}_{i-1}},\vv{\mathsf{w}_{i}}]$.
\end{proof}


Having proved the above lemma, let us remark that R\"{o}nnholm  introduced in \cite{Ronnholm18} the so-called {\em inclusion  quantifiers} $(\exists \vec{w}\subseteq\vec{x})$ and $(\forall \vec{w}\subseteq\vec{x})$, which are closely related to our $\sigma_n$ and $\pi_n$ formulas. The existential inclusion quantification $(\exists \vec{w}\subseteq\vec{x})\phi$ is defined to be true on a team $X$ in a model $M$,  if and only if there exists a sequence  $\vec{F}$ of functions for the standard existential quantifications $\exists\vec{w}$  such that the team $X(\vec{F}/\vec{x})$ generated by these functions should satisfy $\phi$ and also  respect the inclusion atom $\vec{w}\subseteq\vec{x}$, namely, $M\models_{X(\vv{F}/\vec{w})}(\vec{w}\subseteq\vec{x})\wedge \phi$. 
It is then straightforward to see 
that 
\[\exists \mathsf{w}_{i,j}(\mathsf{inc}(\mathsf{w}_{i,j};\vec{x})\wedge\phi)\equiv(\exists \mathsf{w}_{i,j}\subseteq\vec{x})\phi.\]
On the other hand, a universal inclusion quantification $(\forall \vec{w}\subseteq\vec{x})\phi$ is defined to be true on a team $X$, if and only if the team $X(X[\vec{x}]/\vec{w})$ satisfies $\phi$. It then follows  from \Cref{inc_pro_fact}(ii) that 
\[(\forall \mathsf{w}_{i,1}\subseteq\vec{x})\dots(\forall \mathsf{w}_{i,k_i}\subseteq\vec{x})\phi\models\exists \vv{\mathsf{w}_{i}}(\mathsf{pro}(\vv{\mathsf{w}_1},\dots,\vv{\mathsf{w}_{i-1}};\vec{x};\vv{\mathsf{w}_{i}}) \wedge \phi).\]
Without going into further detail we point out that our $\sigma_n$ and $\pi_n$ formulas can be expressed in terms of the inclusion quantifiers, as, e.g., 
\[
\sigma_2[\vec{x},\phi_R(\mathsf{w}_1,\mathsf{w}_2)]\equiv(\exists \mathsf{w}_{1}\subseteq \vec{x})(\forall \mathsf{w}_{2}\subseteq \vec{x})\phi_R(\mathsf{w}_1,\mathsf{w}_2).
\]
We refer the reader to Section 3.3.3 of the  dissertation of R\"{o}nnholm \cite{Ronnholm_thesis} for further discussions, and now we turn to the proof of \Cref{sigma_pi_df_ind}.

\begin{proof}[Proof of \Cref{sigma_pi_df_ind}]
We only give the detailed proof for $\Sigma_{n,\mathsf{k}}^R(x_1,\dots,x_m)$ when $n$ is odd. The other cases can be proved analogously.

Suppose $M\models_X\sigma_n[\vec{x},\phi_R(\vv{\mathsf{w}_1},\dots,\vv{\mathsf{w}_n})]$ for some model $M$ and nonempty team $X$, where 
\begin{align*}
 \sigma_n[\vec{x},\phi_R(\vv{\mathsf{w}_1},\dots,\vv{\mathsf{w}_n})]:=\exists\vv{\mathsf{w}_1}\Big( \mathsf{inc}(\vv{\mathsf{w}_{1}};\vec{x})&\wedge \exists \vv{\mathsf{w}_2}\Big(\mathsf{pro}(\vv{\mathsf{w}_1};\vec{x};\vv{\mathsf{w}_{2}})
\wedge\cdots\,\cdots\wedge\\
&~~~~\exists\vv{\mathsf{w}_n}\Big(  \mathsf{inc}(\vv{\mathsf{w}_{n}};\vec{x})\wedge\phi_R(\vv{\mathsf{w}_1},\dots,\vv{\mathsf{w}_n})\underbrace{\Big)\cdots\Big)\Big)}_{n}.
\end{align*}
Let $\vv{F_1},\dots,\vv{F_n}$ be a sequence of functions for the existential quantifications $\exists \vv{\mathsf{w}_1},\dots,\exists\vv{\mathsf{w}_n}$ in $\sigma_n$ and $Y=X(\vv{F_1},\dots,\vv{F_n}/\vv{\mathsf{w}_1},\dots,\vv{\mathsf{w}_n})$ such that  $M\models_Y\mathsf{inc}(\vv{\mathsf{w}_{i}};\vec{x})$ for any odd $i\leq n$, $M\models_Y\mathsf{pro}(\vv{\mathsf{w}_{1}},\dots,\vv{\mathsf{w}_{i}};\vec{x};\vv{\mathsf{w}_{i+1}})$ for any odd  $i<n$,  
and $M\models_Y\phi_R(\vv{\mathsf{w}_1},\dots,\vv{\mathsf{w}_n})$.


To show that $M\models_X\Sigma_{n,\mathsf{k}}^R(\vec{x})$, take any $t\in Y$. Since $M\models_Y \mathsf{inc}(\mathsf{w}_{1,1},\dots,\mathsf{w}_{1,k_1};\vec{x})$, there exist $s_{1,1},\dots,s_{1,k_1}\in X$ such that 
\[s_{1,1}(\vec{x})=t(\mathsf{w}_{1,1}),\dots,s_{1,k_1}(\vec{x})=t(\mathsf{w}_{1,k_1}).\]
Let $s_{2,1},\dots,s_{2,k_2}\in X$ be arbitrary. Since $M\models_Y\mathsf{pro}(\vv{\mathsf{w}_1};\vec{x};\mathsf{w}_{2,1},\dots,\mathsf{w}_{2,k_2})$,  we have by \Cref{inc_pro_fact}(ii) that 
\[Y(Y[\vec{x}]/\mathsf{w}_{2,1},\dots,\mathsf{w}_{2,k_2})[\vv{\mathsf{w}_1},\vv{\mathsf{w}_2}]\subseteq Y[\vv{\mathsf{w}_1},\vv{\mathsf{w}_2}],\] 
which implies that 
there exist $t_2\in Y$ such that 
\[t_2(\vv{\mathsf{w}_1})=t(\vv{\mathsf{w}_1})=\vv{s_{1}}(\vec{x})\text{ and }t_2(\mathsf{w}_{2,1})=s_{2,1}(\vec{x}),\dots,t_2(\mathsf{w}_{2,k_2})=s_{2,k_2}(\vec{x}).\]

Repeat the argument $n$ times to find in the same manner the  assignments $\vv{s_{3}}\in X^{k_3},\vv{s_{5}}\in X^{k_5},\dots,\vv{s_{n}}\in X^{k_n}$ and  the corresponding assignments $t_4,t_6,\dots,t_{n-1}\in Y$ for arbitrary $\vv{s_{4}}\in X^{k_4},\vv{s_{6}}\in X^{k_6},\dots,\vv{s_{n-1}}\in X^{k_{n-1}}$. In the last step we have 
\[t_{n-1}(\vv{\mathsf{w}_1})=\vv{s_{1}}(\vec{x}),\dots,t_{n-1}(\vv{\mathsf{w}_{n-1}})=\vv{s_{n-1}}(\vec{x})\]
and there exist $s_{n,1},\dots,s_{n,k_n}\in X$ such that 
\[s_{n,1}(\vec{x})=t_{n-1}(\mathsf{w}_{n,1}),\dots,s_{n,k_n}(\vec{x})=t_{n-1}(\mathsf{w}_{n,k_n}).\]
Since $M\models_Y\phi_R(\vv{\mathsf{w_1}},\dots,\vv{\mathsf{w_n}})$, we have $M\models_{\{t_{n-1}\}}\phi_R(\vv{\mathsf{w_1}},\dots,\vv{\mathsf{w_n}})$ by the downward closure property of first-order formulas. Since the  formula $\phi_R$ defines $R$, we conclude \[(t_{n-1}(\vv{\mathsf{w}_1}),\dots,t_{n-1}(\vv{\mathsf{w}_n}))\in R^M\text{, thereby }
(\vv{s_1}(\vec{x}),\dots,\vv{s_{n}}(\vec{x}))\in R^M.\]

\vspace{0.5\baselineskip}

Conversely, suppose $M\models_X\Sigma_{n,\mathsf{k}}^R(\vec{x})$ for some model $M$ and nonempty team $X$. Then
\begin{equation}\label{hier_lm_eq1}
(\exists \vv{s_1}\in X^{k_1})(\forall\vv{s_2}\in X^{k_2})\cdots \cdots( \exists \vv{s_{n}}\in X^{k_n}) (\vv{s_1}(\vec{x}),\dots,\vv{s_{n}}(\vec{x}))\in R^M.
\end{equation}

To show that $M\models_X\sigma_n[\vec{x},\phi_R(\vv{\mathsf{w}_1},\dots,\vv{\mathsf{w}_n})]$, 
consider $s_{1,1},\dots,s_{1,k_1}\in X$. One can define a sequence $\vv{F_{1,1}},\dots,\vv{F_{1,k_1}}$ of functions for the existential quantifications $\exists \mathsf{w}_{1,1}\dots\exists \mathsf{w}_{1,k_1}$ in $\sigma_n$ such that 
\[X(\vv{F_{1,1}},\dots,\vv{F_{1,k_1}}/\mathsf{w}_{1,1},\dots,\mathsf{w}_{1,k_1})=X(s_{1,1}(\vec{x}),\dots,s_{1,k_1}(\vec{x})/\mathsf{w}_{1,1},\dots,\mathsf{w}_{1,k_1}).
\]
Put $Y_1=X(\vv{F_{1,1}},\dots,\vv{F_{1,k_1}}/\mathsf{w}_{1,1},\dots,\mathsf{w}_{1,k_1})$. 
Clearly,  we have $M\models_{Y_1}\mathsf{w}_{1,j}\subseteq \vec{x}$ for all $1\leq j\leq k_1$, namely, $M\models_{Y_1}\mathsf{inc}(\vv{\mathsf{w}_{1}};\vec{x})$. It then remains to show that $M\models_{Y_1} \pi_{n-1}[\vec{x},\phi_R(\vv{\mathsf{w}_1},\dots,\vv{\mathsf{w}_n})]$.

Now, consider the quantifications $\exists\mathsf{w_{2,2}}\dots\exists\mathsf{w_{2,k_2}}$ in $\pi_{n-1}$. One can define a sequence $\vv{F_{2,1}},\dots,\vv{F_{2,k_2}}$ of functions such that  
\begin{equation}\label{hier_lm_eq24}
Y_1(\vv{F_{2,1}},\dots,\vv{F_{2,k_2}}/\mathsf{w}_{2,1},\dots,\mathsf{w}_{2,k_2})=Y_1(X[\vec{x}]/\mathsf{w}_{2,1},\dots,\mathsf{w}_{2,k_2}). 
\end{equation}
Put $Y_2=Y_1(\vv{F_{2,1}},\dots,\vv{F_{2,k_2}}/\mathsf{w}_{2,1},\dots,\mathsf{w}_{2,k_2})$. 
Since
\(Y_2(Y_2[\vec{x}]/\vv{\mathsf{w}_{2}})=Y_1(X[\vec{x}]/\vv{\mathsf{w}_{2}})=Y_2,\)
we obtain by \Cref{inc_pro_fact}(ii) that  $M\models_{Y_2}\mathsf{pro}(\vv{\mathsf{w}_1};\vec{x};\vv{\mathsf{w}_{2}})$. It then remains to show that $M\models_{Y_2} \sigma_{n-2}[\vec{x};\phi_R(\vv{\mathsf{w}_1},\dots,\vv{\mathsf{w}_n})]$. 

For each $t\in Y_2$, by  (\ref{hier_lm_eq24}), there exists $\vv{s_{t,2}}=\langle s^t_{2,1},\dots,s^t_{2,k_2}\rangle\in X^{k_2}$  such that 
\[s^t_{2,1}(\vec{x})=t(\mathsf{w}_{2,1}),\dots,s^t_{2,k_2}(\vec{x})=t(\mathsf{w}_{2,k_2}).\]
Hence, by (\ref{hier_lm_eq1}), there exists  $\vv{s_{t,3}}=\langle s^t_{3,1},\dots,s^t_{3,k_3}\rangle\in X^{k_3}$ such that
\[(\forall\vv{s_4}\in X^{k_4})\cdots \cdots( \exists \vv{s_{n}}\in X^{k_n}) (\vv{s_1}(\vec{x}),\vv{s_{t,2}}(\vec{x}),\vv{s_{t,3}}(\vec{x}),\vv{s_4}(\vec{x})\dots,\vv{s_{n}}(\vec{x}))\in R^M.\]
One can define a sequence $\vv{F_{3,1}},\dots,\vv{F_{3,k_3}}$ of functions for the existential quantifications $\exists\mathsf{w}_{3,1}\dots\exists\mathsf{w}_{3,k_3}$ in $\sigma_{n-2}$ such that for all $t\in Y_3=Y_2(\vv{F_{3,1}},\dots,\vv{F_{3,k_3}}/\mathsf{w}_{3,1},\dots,\mathsf{w}_{3,k_3})$,
\[t(\mathsf{w}_{3,1})=s^{t}_{3,1}(\vec{x}),\dots,t(\mathsf{w}_{3,k_3})=s^{t}_{3,k_3}(\vec{x}),\]
where we use the same notation $t$ to denote also  $t\upharpoonright {\rm dom}(Y_2)\in Y_2$. 
This definition implies immediately that  $M\models_{Y_3}\mathsf{inc}(\vv{\mathsf{w}_{3}};\vec{x})$.
It then remains to show that $M\models_{Y_3} \pi_{n-3}[\vec{x};\phi_R(\vv{\mathsf{w}_1},\dots,\vv{\mathsf{w}_n})]$. 


For the existential quantifications $\exists\mathsf{w}_{4,1}\dots\exists\mathsf{w}_{4,k_4}$ in $\pi_{n-3}$, we construct the sequence $\vv{F_{4,1}},\dots,\vv{F_{4,k_4}}$  of functions such that $Y_4=Y_3(\vv{F_{4,1}},\dots,\vv{F_{4,k_4}}/\mathsf{w}_{4,1},\dots,\mathsf{w}_{4,k_4})=Y_3(X[\vec{x}]/\vv{\mathsf{w}_4})$, and proceed in the same way as above.
Repeat the same argument $n$ times. In the last step we have the team $Y_n$  defined and $M\models_{Y_n}\mathsf{inc}(\vv{\mathsf{w}_{n}};\vec{x})$. It then only remains to show that $M\models_{Y_n}\phi_R(\vv{\mathsf{w}_1},\dots,\vv{\mathsf{w}_n})$.
Since $\phi_R$ is flat, this reduces to showing that $M\models_{\{t\}}\phi_R$ holds for any $t\in Y_n$. By the definition of $Y_n$, we have
\[(\vv{s_1}(\vec{x}),\vv{s_{t,2}}(\vec{x}),\vv{s_{t,3}}(\vec{x}),\vv{s_{t,4}}(\vec{x})\dots,\vv{s_{t,n}}(\vec{x}))\in R^M\]
\[\text{and }
t(\vv{\mathsf{w}_1})=\vv{s_1}(\vec{x}),~t(\vv{\mathsf{w}_2})=\vv{s_{t,2}}(\vec{x}),~\dots,~t(\vv{\mathsf{w}_n})=\vv{s_{t,n}}(\vec{x}),\]
which imply $M\models_{\{t\}}\phi_R(\vv{\mathsf{w}_1},\dots,\vv{\mathsf{w}_n})$, as the first-order formula $\phi_R$ defines $R$.
\end{proof}


Having identified a hierarchy of negatable formulas, a natural question to ask is whether it is possible to extend the hierarchy in some way while keeping the negatability of the formulas. Let us now try to provide an answer to this question. 


Recall that the $\Sigma^1_1$-translations for the disjunction $\cor$, the existential quantifier $\exists$ and universal quantifier $\forall$  as given in \cite{Van07dl} are as follows:
\begin{itemize}
\item $\tau_{\phi\cor \psi}(R)=\exists S\exists S'(\tau_\phi(S)\wedge \tau_\psi(S')\wedge \forall \vec{x}(R\vec{x}\to (S\vec{x}\vee S'\vec{x})))$,
\item $\tau_{\exists x\phi(x,\vec{y})}(R)=\exists S(\tau_\phi(S)\wedge\forall \vec{y}(R\vec{y}\to \exists x Sx\vec{y}))$,
\item $\tau_{\forall x\phi(x,\vec{y})}(R)=\exists S(\tau_\phi(S)\wedge\forall \vec{y}(R\vec{y}\to \forall x Sx\vec{y}))$.
\end{itemize}
These translations can well be non-first-order sentences, and thus none of $\cor$, $\exists$ and $\forall$ preserve negatability. Nevertheless, in the literature there are other variants of these logical constants, under which the set of negatable formulas is closed, such as the disjunction $\ior$ (introduced in \cite{AbVan09}, known in the literature by the name \emph{intuitionistic disjunction} or \emph{Boolean disjunction}), the (weak) existential quantifier $\exists^1$ and the (weak) universal quantifier $\forall^1$ (introduced in \cite{KontVan09}) whose semantics are defined as follows:
\begin{itemize}
 \item {\em $M\models_X\phi\ior\psi$ ~iff~ $M\models_X\phi$ or $M\models_X\psi$.}
 \item {\em $M\models_X\exists^1x\phi$ ~iff~  $M\models_{X(a/x)}\phi$ for some $a\in M$.}
\item {\em $M\models_X\forall^1x\phi$ ~iff~ $M\models_{X(a/x)}\phi$ for all $a\in M$.}
\end{itemize}
Clearly, for all formulas $\phi$ of \D or \Ind, all models $M$ and all teams $X$,
\[M\models_X\phi\ior\psi\iff (M,rel(X))\models\tau_\phi(R)\vee\tau_\psi(R),\]
\[M\models_X\exists^1x\phi\iff (M,rel(X))\models\exists x\tau_\phi(R)\]
and
\[M\models_X\forall^1x\phi\iff (M,rel(X))\models\forall x\tau_\phi(R).\]
In view of this, \Cref{ind2sigma11} can be generalized to cover extensions of \D and \Ind with these three logical constants, and these three logical constants are then definable in \D and in \Ind \footnote{Moreover, $\exists^1$ and $\ior$ are uniformly definable in \D and in \Ind, since  $\exists^1x\phi\equiv \exists x(\dep(x)\wedge\phi)$ and 
\(\phi\ior\psi\equiv \exists w\exists u(\dep(w)\wedge \dep(u)\wedge (w=u\,\cor \phi)\wedge (w\neq u\,\cor \psi))\), where $w,u$ are fresh variables.}.
If $\tau_\phi(R)$ and $\tau_\psi(R)$ are both first-order, then the above three $\Sigma^1_1$-translations $\tau_\phi(R)\vee\tau_\psi(R)$, $\exists x\tau_\phi(R)$ and $\forall x\tau_\phi(R)$ are first-order as well. By \Cref{neg_char_thm_so}, this shows that the logical constants $\ior$, $\exists^1$ and $\forall^1$ preserve negatability. Furthermore, it is easy to see that
\begin{itemize}
\item $\cn(\phi\wedge\psi)\equiv \cn\phi\ior \cn\psi$, \quad $\cn(\phi\ior\psi)\equiv \cn\phi\wedge \cn\psi$,
\item $\cn\exists^1x\phi\equiv\forall^1x\cn\phi$~~~ and ~~~$\cn\forall^1x\phi\equiv\exists^1x\cn\phi$.
\end{itemize}

Without going into detail we remark that making use of the  disjunction $\ior$ the extended system of \Ind can be applied to give a new formal proof of Arrow's Impossibility Theorem \cite{Arrow51} in social choice theory.  In \cite{PacuitYang2016} the theorem is formulated as an entailment $\Gamma_{\textsf{Arrow}}\models\phi_{\textsf{dictator}}$ in independence logic, where $\Gamma_{\textsf{Arrow}}$ is a set of formulas expressing  the conditions in Arrow's Impossibility Theorem and $\phi_{\textsf{dictator}}$ is a formula expressing the existence of a dictator. The formula $\phi_{\textsf{dictator}}$ is of the form $\bigior_{i=1}^n\phi_i$, where $\phi_i$ is a first-order formula expressing that voter $i$ is a dictator (among $n$ voters).
By what we just obtained, the formula $\phi_{\textsf{dictator}}$ is negatable in \Ind and the Completeness Theorem guarantees that $\Gamma_{\textsf{Arrow}}\vdash_{\Ind}^\ast\phi_{\textsf{dictator}}$ is derivable in our extended system.

\section{Armstrong's axioms and the Geiger-Paz-Pearl axioms}

Dependence and independence atoms, being members of the hierarchy  defined in the previous section (see \Cref{neg_atom_example}), are negatable in independence logic. Therefore, Armstrong's Axioms \cite{Armstrong_Axioms} that characterize dependence atoms and the Geiger-Paz-Pearl axioms \cite{Geiger-Paz-Pearl_1991} that characterize independence atoms are derivable in our extended system of independence logic. In this section, we provide the derivations of these axioms in order to illustrate the power of our extended system.

Throughout this section we denote by $\vdash$ the syntactic consequence relation associated with the extended system of \Ind defined in Section 4. The crucial rules from \cite{Hannula_fo_ind_13} that we will apply in our derivations are $\exists$\text{I}, $\exists$\textsf{E}, the usual rules for identity ($=$) and the following ones, where we write $\mathsf{x}, \mathsf{y}, \mathsf{z},\dots$ for arbitrary  sequences of variables:
\begin{center}
\renewcommand{\arraystretch}{1.8}
\def\ScoreOverhang{0.5pt}
\def\defaultHypSeparation{\hskip .1in}
\begin{tabular}{|C{0.45\linewidth}|C{0.45\linewidth}|}
\hline
\AxiomC{}\RightLabel{$\subseteq$\textsf{Id}}\UnaryInfC{$\mathsf{x}\subseteq\mathsf{x}$}\DisplayProof
&\AxiomC{}\noLine\UnaryInfC{$\mathsf{x}\subseteq\mathsf{y}$}\AxiomC{$\mathsf{y}\subseteq\mathsf{z}$}\RightLabel{$\subseteq$\textsf{Trs}}\BinaryInfC{$\mathsf{x}\subseteq\mathsf{z}$}\noLine\UnaryInfC{}\DisplayProof\\\hline
\AxiomC{}\noLine\UnaryInfC{$x_1\dots x_n\subseteq y_1\dots y_n$}\RightLabel{$\subseteq$\textsf{Pro}}\UnaryInfC{$x_{i_1}\dots x_{i_k}\subseteq y_{i_1}\dots y_{i_k}$}\DisplayProof&\AxiomC{}\noLine\UnaryInfC{$\mathsf{y}\subseteq\mathsf{x}$}\AxiomC{$\alpha(\mathsf{x})$}\RightLabel{$\subseteq$\textsf{Cmp}}\BinaryInfC{$\alpha(\mathsf{y}/\mathsf{x})$}\DisplayProof\\[-2pt]
{\footnotesize where $\{i_1,\dots,i_k\}\subseteq \{1,\dots,n\}$.}&{\footnotesize where $\alpha$ is first-order.}\\\hline
\multicolumn{2}{|C{0.96\linewidth}|}{\AxiomC{}\noLine\UnaryInfC{$\mathsf{x}\perp_{\mathsf{z}}\mathsf{y}$} \AxiomC{$\mathsf{w}_1\mathsf{u}_1\mathsf{v}_1\mathsf{t}_1\subseteq \mathsf{x}\mathsf{y}\mathsf{z}\mathsf{s}$}\AxiomC{$\mathsf{w}_2\mathsf{u}_2\mathsf{v}_2\mathsf{t}_2\subseteq \mathsf{x}\mathsf{y}\mathsf{z}\mathsf{s}$} \RightLabel{$\perp$\textsf{E}}\TrinaryInfC{$\exists \mathsf{w}_3\mathsf{u}_3\mathsf{v}_3\mathsf{t}_3\big(\mathsf{w}_3\mathsf{u}_3\mathsf{v}_3\mathsf{t}_3\subseteq \mathsf{x}\mathsf{y}\mathsf{z}\mathsf{s}\wedge (\mathsf{v}_1=\mathsf{v}_2\to \mathsf{w}_3\mathsf{u}_3\mathsf{v}_3=\mathsf{w}_1\mathsf{u}_2\mathsf{v}_2)\big)$}\DisplayProof}\\
\multicolumn{2}{|C{0.96\linewidth}|}{\footnotesize where $\mathsf{v}_1=\mathsf{v}_2\to \mathsf{w}_3\mathsf{u}_3\mathsf{v}_3=\mathsf{w}_1\mathsf{u}_2\mathsf{v}_2$ is short for $\mathsf{v}_1\neq\mathsf{v}_2\vee \mathsf{w}_3\mathsf{u}_3\mathsf{v}_3=\mathsf{w}_1\mathsf{u}_2\mathsf{v}_2$.}\\\hline
\end{tabular}
\end{center}
To simplify the derivations, we will also use the following handy  (sound) weakening rule for inclusion atoms that was introduced essentially in \cite{HannulaKontinen16}: 

\begin{center}
\renewcommand{\arraystretch}{1.8}
\def\ScoreOverhang{0.5pt}
\def\defaultHypSeparation{\hskip .1in}
\begin{tabular}{|C{0.45\linewidth}|C{0.45\linewidth}|}
\hline
\multicolumn{2}{|C{0.96\linewidth}|}{ \AxiomC{}\noLine\UnaryInfC{$\mathsf{x}\subseteq\mathsf{y}$}\RightLabel{$\subseteq$\textsf{W}}\UnaryInfC{$\exists \mathsf{w}(\mathsf{x}\mathsf{w}\subseteq\mathsf{y}\mathsf{z})$}\noLine\UnaryInfC{} \DisplayProof}\\\hline
\end{tabular}
\end{center}
By applying the rule $\subseteq$\textsf{W} we can easily derive a more general version of the $\perp$\textsf{E} rule:
\[\AxiomC{}\noLine\UnaryInfC{$\mathsf{x}\mathsf{x}'\perp_{\mathsf{z}}\mathsf{y}\mathsf{y}'$} \AxiomC{$\mathsf{w_1}\mathsf{v_1}\subseteq \mathsf{x}\mathsf{z}$}\AxiomC{$\mathsf{u_2}\mathsf{v_2}\subseteq \mathsf{y}\mathsf{z}$} \RightLabel{$\perp$\textsf{E}}\TrinaryInfC{$\exists \mathsf{w_3}\mathsf{u_3}\mathsf{v}_3\mathsf{t}_3\big(\mathsf{w_3}\mathsf{u_3}\mathsf{v}_3\mathsf{t}_3\subseteq \mathsf{x}\mathsf{y}\mathsf{z}\mathsf{s}\wedge (\mathsf{v_1}=\mathsf{v_2}\to \mathsf{w_3}\mathsf{u_3}\mathsf{v}_3=\mathsf{w_1}\mathsf{u_2}\mathsf{v_2})\big)$}\DisplayProof\]


We first derive Armstrong's Axioms. To state these axioms in full generality, we now introduce a generalized version of dependence atoms, atoms $\dep(\mathsf{x},\mathsf{y})$ that can have a sequence of variables in the last coordinate. The semantics  of $\dep(\mathsf{x},\mathsf{y})$ is defined as:
\begin{itemize}
\item {\em $M\models_X\dep(\mathsf{x},\mathsf{y})$ ~iff~ for all $s,s'\in X$, if $s(\mathsf{x})=s'(\mathsf{x})$, then $s(\mathsf{y})=s'(\mathsf{y})$.}
\end{itemize}
Clearly, $\dep(\mathsf{x},\mathsf{y})\equiv\mathsf{y}\perp_{\mathsf{x}}\mathsf{y}$, and we thus interpret the dependence atom $\dep(\mathsf{x},\mathsf{y})$ in \Ind as  the independence atom $\mathsf{y}\perp_{\mathsf{x}}\mathsf{y}$. 

\begin{exmp}
The following clauses, known as Armstrong's Axioms \cite{Armstrong_Axioms}, are derivable in the extended system of \Ind.
\begin{enumerate}[(1)]
\item $\vdash\dep(\mathsf{x},\mathsf{x})$
\item $\dep(\mathsf{x},\mathsf{y},\mathsf{z})\vdash\dep(\mathsf{y},\mathsf{x},\mathsf{z})$
\item $\dep(\mathsf{x},\mathsf{x},\mathsf{y})\vdash\dep(\mathsf{x},\mathsf{y})$
\item $\dep(\mathsf{y},\mathsf{z})\vdash\dep(\mathsf{x},\mathsf{y},\mathsf{z})$
\item $\dep(\mathsf{x},\mathsf{y}),\dep(\mathsf{y},\mathsf{z})\vdash\dep(\mathsf{x},\mathsf{z})$
\end{enumerate}
\end{exmp}
\begin{proof}
We only give the detailed derivation for item (5). By the rule \cn\!\textsf{E}, it suffices to derive $\dep(\mathsf{x},\mathsf{y}),\dep(\mathsf{y},\mathsf{z}),\cn\dep(\mathsf{x},\mathsf{z})\vdash\bot$, which, 
by the translation given in \Cref{sigma_pi_df_ind} (see also Formula (\ref{neg_dep_frm})), 
is equivalent to
\[
\dep(\mathsf{x},\mathsf{y}),\dep(\mathsf{y},\mathsf{z}),\exists \mathsf{w}_1 \mathsf{v}_1\exists \mathsf{w}_2 \mathsf{v}_2\big( \mathsf{w}_1  \mathsf{v}_1\subseteq \mathsf{x} \mathsf{z}\,\wedge\,  \mathsf{w}_2  \mathsf{v}_2\subseteq \mathsf{x} \mathsf{z}\,\wedge\,  \mathsf{w}_1= \mathsf{w}_2\,\wedge\, \mathsf{v}_1\neq  \mathsf{v}_2  \big)\vdash\bot.
\]
By $\subseteq$\textsf{W} and $\subseteq$\textsf{Pro}, it is sufficient to derive 
\begin{align*}
\dep(\mathsf{x},\mathsf{y}),\dep(\mathsf{y},\mathsf{z}),&\exists \mathsf{w}_1  \mathsf{u}_1\mathsf{v}_1\exists \mathsf{w}_2  \mathsf{u}_2\mathsf{v}_2\\
&\big( \mathsf{w}_1  \mathsf{u}_1 \mathsf{v}_1\subseteq \mathsf{x}  \mathsf{y}\mathsf{z}\,\wedge\,  \mathsf{w}_2  \mathsf{u}_2 \mathsf{v}_2\subseteq \mathsf{x} \mathsf{y} \mathsf{z}\,\wedge\,  \mathsf{w}_1= \mathsf{w}_2\,\wedge\, \mathsf{v}_1\neq  \mathsf{v}_2  \big)\vdash\bot.
\end{align*}
Further, by $\exists$\textsf{E}, it suffices to prove
\begin{equation}\label{arms_deriv}
\dep(\mathsf{x},\mathsf{y}),\dep(\mathsf{y},\mathsf{z}), \mathsf{w}_1 \mathsf{u}_1\mathsf{v}_1\subseteq \mathsf{x}\mathsf{y}\mathsf{z}, \mathsf{w}_2 \mathsf{u}_2\mathsf{v}_2\subseteq \mathsf{x}\mathsf{y}\mathsf{z}\vdash  \mathsf{w}_1=\mathsf{w}_2\to \mathsf{v}_1=\mathsf{v}_2.
\end{equation}

First, note that since $\vdash\mathsf{y}=\mathsf{y}$, by $\subseteq$\textsf{Cmp} we have 
\begin{equation}\label{arms_deriv_eq1}
\mathsf{u}\mathsf{u}'\subseteq \mathsf{y}\mathsf{y}\vdash\mathsf{u}=\mathsf{u}'.
\end{equation}
 Now, we derive
\begin{align*}
&\mathsf{y}\perp_{\mathsf{x}}\mathsf{y},\mathsf{w}_1 \mathsf{u}_1 \mathsf{v}_1\subseteq \mathsf{x}\mathsf{y}\mathsf{z}, \mathsf{w}_2 \mathsf{u}_2\mathsf{v}_2\subseteq \mathsf{x}\mathsf{y}\mathsf{z}\\
\vdash~&\exists\mathsf{w}\mathsf{u}\mathsf{u}'\mathsf{v}\big(\mathsf{w}\mathsf{u}\mathsf{u}'\mathsf{v}\subseteq\mathsf{x}\mathsf{y}\mathsf{y}\mathsf{z}\wedge(\mathsf{w}_1=\mathsf{w}_2\to \mathsf{w}\mathsf{u}\mathsf{u}'=\mathsf{w}_1\mathsf{u}_1\mathsf{u}_2)\big)\tag{by the generalized $\perp$\textsf{E}}\\
\vdash~&\mathsf{w}_1=\mathsf{w}_2\to  \mathsf{u}_1=\mathsf{u}_2.\tag{by $\subseteq$\textsf{Pro} and (\ref{arms_deriv_eq1})}
\end{align*}
Similarly, we can also derive
\[\dep(\mathsf{y},\mathsf{z}),\mathsf{w}_1 \mathsf{u}_1\mathsf{v}_1\subseteq \mathsf{x}\mathsf{y}\mathsf{z}, \mathsf{w}_2 \mathsf{u}_2\mathsf{v}_2\subseteq \mathsf{x}\mathsf{y}\mathsf{z}\vdash \mathsf{u}_1=\mathsf{u}_2\to \mathsf{v}_1=\mathsf{v}_2.\]
Hence, (\ref{arms_deriv}) follows (from the usual rules for first-order formulas that the system of \cite{Hannula_fo_ind_13} contains).
\end{proof}

Next, we derive the Geiger-Paz-Pearl axioms as follows.

\begin{exmp}
The following clauses, known as the Geiger-Paz-Pearl axioms \cite{Geiger-Paz-Pearl_1991}, are derivable in the extended system of \Ind.
\begin{enumerate}[(1)]
\item $\mathsf{x}\perp\mathsf{y}~\vdash\mathsf{y}\perp\mathsf{x}$
\item $\mathsf{x}\perp\mathsf{y}~\vdash\mathsf{z}\perp\mathsf{y}$, where $\mathsf{z}$ is a subsequence of $\mathsf{x}$.
\item $\mathsf{x}\perp\mathsf{y}~\vdash\mathsf{u}\perp\mathsf{v}$, where $\mathsf{u}$ and $\mathsf{v}$ are  permutations of $\mathsf{x}$ and  $\mathsf{y}$, respectively
\item $\mathsf{x}\perp \mathsf{y},~\mathsf{x}\mathsf{y}\perp\mathsf{z}~\vdash\mathsf{x}\perp \mathsf{y}\mathsf{z}$.
\end{enumerate}
\end{exmp}
\begin{proof}
We only give the detailed derivation for item (4). By the rule \cn\!\textsf{E}, it suffices to derive $\mathsf{x}\perp \mathsf{y},\mathsf{x}\mathsf{y}\perp\mathsf{z},\cn(\mathsf{x}\perp \mathsf{y}\mathsf{z})\vdash\bot$. Instead of applying \Cref{sigma_pi_df_ind}, 
we will now use the more succinct definition of $\cn(\mathsf{x}\perp \mathsf{y}\mathsf{z})$ as given in Equation (\ref{neg_ind_smpdf}), and the desired clause is then equivalent to
\begin{equation*}
\begin{split}
&\mathsf{x}\perp \mathsf{y},\mathsf{x}\mathsf{y}\perp\mathsf{z},\exists^1 \mathsf{w} \mathsf{u}\mathsf{v}( \mathsf{w}\subseteq \mathsf{x}\,\wedge \,  \mathsf{u}\mathsf{v}\subseteq \mathsf{y}\mathsf{z}\,\wedge\, \mathsf{w} \mathsf{u}\mathsf{v}\neq \mathsf{x}\mathsf{y}\mathsf{z} )\vdash\bot.
\end{split}
\end{equation*}
By $\exists$\textsf{E}, it suffices to derive
\begin{equation}\label{gpp_deriv}
\Gamma,\mathsf{x}\perp \mathsf{y},\mathsf{x}\mathsf{y}\perp\mathsf{z}\,, \mathsf{w}\subseteq \mathsf{x}\, , \mathsf{u}\mathsf{v}\subseteq \mathsf{y}\mathsf{z}\, , \mathsf{w} \mathsf{u}\mathsf{v}\neq \mathsf{x}\mathsf{y}\mathsf{z}\vdash\bot,
\end{equation}
where $\Gamma=\{\dep(o)\mid o\text{ is a variable from }\mathsf{w}\mathsf{u}\mathsf{v}\}$.

First, we derive by (the general version of) $\perp$\textsf{E} that
\begin{equation}\label{gpp_deriv_1}
\mathsf{x}\perp \mathsf{y},\,\mathsf{w}\subseteq \mathsf{x},\,  \mathsf{u}\mathsf{v}\subseteq \mathsf{y}\mathsf{z}\vdash\exists \mathsf{w}_1 \mathsf{u}_1(\mathsf{w}_1 \mathsf{u}_1\subseteq \mathsf{x}\mathsf{y}\wedge \mathsf{w}_1 \mathsf{u}_1=\mathsf{w}\mathsf{u}).
\end{equation}
By $\perp$\textsf{E} again, we derive
\begin{equation}\label{gpp_deriv_2}
\mathsf{xy}\perp \mathsf{z},\,\mathsf{w}_1 \mathsf{u}_1\subseteq \mathsf{x}\mathsf{y},\,  \mathsf{u}\mathsf{v}\subseteq \mathsf{y}\mathsf{z}\vdash\exists \mathsf{w}_2\mathsf{u}_2\mathsf{v}_2(\mathsf{w}_2 \mathsf{u}_2\mathsf{v}_2\subseteq \mathsf{x}\mathsf{y}\mathsf{z}\wedge \mathsf{w}_2 \mathsf{u}_2\mathsf{v}_2=\mathsf{w}_1\mathsf{u}_1\mathsf{v}).
\end{equation}
Putting (\ref{gpp_deriv_1}) and (\ref{gpp_deriv_2}) together, by $\exists$\textsf{E} we obtain
\[
\mathsf{x}\perp \mathsf{y},\mathsf{x}\mathsf{y}\perp\mathsf{z}\,, \mathsf{w}\subseteq \mathsf{x}\, , \mathsf{u}\mathsf{v}\subseteq \mathsf{y}\mathsf{z}\vdash\exists \mathsf{w}_2 \mathsf{u}_2\mathsf{v}_2(\mathsf{w}_2 \mathsf{u}_2\mathsf{v}_2\subseteq \mathsf{x}\mathsf{y}\mathsf{z}\wedge \mathsf{w}_2 \mathsf{u}_2\mathsf{v}_2=\mathsf{w}\mathsf{u}\mathsf{v})
\vdash \mathsf{w}\mathsf{u}\mathsf{v}\subseteq \mathsf{x}\mathsf{y}\mathsf{z}.
\]

To derive (\ref{gpp_deriv}) it is now sufficient to derive
\begin{equation}\label{gpp_deriv_7}
\Gamma, \, \mathsf{w}\mathsf{u}\mathsf{v}\subseteq \mathsf{x}\mathsf{y}\mathsf{z},\,\mathsf{w}\mathsf{u}\mathsf{v}\neq \mathsf{x}\mathsf{y}\mathsf{z}\vdash\bot.
\end{equation}
First, by $\subseteq$\textsf{W}, we have $\mathsf{w}\mathsf{u}\mathsf{v}\subseteq \mathsf{x}\mathsf{y}\mathsf{z}\vdash\exists \mathsf{w}'\mathsf{u}'\mathsf{v}'(\mathsf{w}'\mathsf{u}'\mathsf{v}'\mathsf{w}\mathsf{u}\mathsf{v}\subseteq \mathsf{w}\mathsf{u}\mathsf{v}\mathsf{x}\mathsf{y}\mathsf{z})$. Moreover, by $\subseteq$\textsf{Cmp}, we have $\mathsf{w}'\mathsf{u}'\mathsf{v}'\mathsf{w}\mathsf{u}\mathsf{v}\subseteq \mathsf{w}\mathsf{u}\mathsf{v}\mathsf{x}\mathsf{y}\mathsf{z},\mathsf{w}\mathsf{u}\mathsf{v}\neq \mathsf{x}\mathsf{y}\mathsf{z}\vdash \mathsf{w}'\mathsf{u}'\mathsf{v}'\neq \mathsf{w}\mathsf{u}\mathsf{v}$. Putting these together,  (\ref{gpp_deriv_7}) is  reduced, by  $\exists$\textsf{E}, to 
\begin{equation*}
\Gamma, \, \mathsf{w}'\mathsf{u}'\mathsf{v}'\mathsf{w}\mathsf{u}\mathsf{v}\subseteq \mathsf{w}\mathsf{u}\mathsf{v}\mathsf{x}\mathsf{y}\mathsf{z},\,\mathsf{w}'\mathsf{u}'\mathsf{v}'\neq \mathsf{w}\mathsf{u}\mathsf{v}\vdash\bot,
\end{equation*}
and by $\subseteq$\textsf{Pro}, further  to
\begin{equation}\label{gpp_deriv_8}
\Gamma, \, \mathsf{w}'\mathsf{u}'\mathsf{v}'\subseteq \mathsf{w}\mathsf{u}\mathsf{v},\,\mathsf{w}'\mathsf{u}'\mathsf{v}'\neq \mathsf{w}\mathsf{u}\mathsf{v}\vdash\bot.
\end{equation}

Now, for an arbitrary variable $o$ from the sequence $\mathsf{wuv}$, we derive the following:
\begin{align*}
o\perp o,\, o\subseteq o\, , \mathsf{w}'\mathsf{u}'\mathsf{v}'\subseteq \mathsf{w}\mathsf{u}\mathsf{v}&\vdash \exists pq(pq\subseteq oo\wedge pq=oo')\tag{by $\perp$\textsf{E}}\\
&\vdash o=o'.\tag{by (\ref{arms_deriv_eq1})}
\end{align*}
Since $\vdash o\subseteq o$ holds by $\subseteq$\textsf{Id}, we conclude that $o\perp o,\mathsf{w}'\mathsf{u}'\mathsf{v}'\subseteq \mathsf{w}\mathsf{u}\mathsf{v}\vdash o=o'$ for any variable $o$ from $\mathsf{wuv}$. Hence, we obtain $\Gamma, \, \mathsf{w}'\mathsf{u}'\mathsf{v}'\subseteq \mathsf{w}\mathsf{u}\mathsf{v}\vdash \mathsf{w}'\mathsf{u}'\mathsf{v}'=\mathsf{w}\mathsf{u}\mathsf{v}$, from which (\ref{gpp_deriv_8}) follows.
\end{proof}

\section{Concluding remarks}


In this paper, we have extended the systems of natural deduction of dependence and independence logic defined in \cite{Axiom_fo_d_KV} and \cite{Hannula_fo_ind_13} and obtained complete axiomatizations of the negatable consequences in these logics. We have also given a characterization of negatable formulas in \Ind and negatable sentences in \D. Determining whether a formula of \Ind or \D is negatable is an undecidable problem. Nevertheless, we identified an interesting class of negatable formulas, the Boolean and weak quantifier closure of the class of $\Sigma_{n,\mathsf{k}}^R$ and $\Pi_{n,\mathsf{k}}^R$  atoms. First-order formulas, dependence and independence atoms belong to this class. 
We also gave derivations of Armstrong's axioms and the Geiger-Paz-Pearl axioms in our extended system of \Ind.


The results of this paper can be generalized in two directions. The first direction is to identify other negatable formulas than those  in the Boolean and weak quantifier closure  of the set of atoms from our hierarchy. The other direction is to analyze the $\Sigma_{n,\mathsf{k}}^R$ and $\Pi_{n,\mathsf{k}}^R$  atoms in more detail. As we saw in \Cref{neg_atom_example}, 
first-order formulas and the atoms of dependence and independence belong  to the $\Pi_1$ or $\Pi_2$ level. 
It is easy to verify that $\Pi_{1,\mathsf{k}}^R$ atoms (including first-order formulas and dependence atoms) are closed downward, $\Sigma_{1,\mathsf{k}}^R$ atoms are closed upward, and $\Pi_{2,\mathsf{k}}^R$ atoms (including inclusion atoms) are closed under unions. First-order logic extended with upward closed atoms is shown in \cite{Galliani2015} to be equivalent to first-order logic. Adding other such atoms to first-order logic results in many new  logics that are less expressive than $\Sigma^1_1$ or independence logic and possibly stronger than first-order logic. These logics are potentially interesting, because, among other properties, by the argument of this paper, the negatable consequences in these logics can in principle be effectively axiomatized.



\section*{Acknowledgements} 
The author is in debted to Eric Pacuit for our discussions on formalizing Arrow's Theorem in independence logic, which in the end  led to the results of this paper unexpectedly. The author would also like to thank Juha Kontinen, Raine R\"{o}nnholm and Jouko V\"{a}\"{a}n\"{a}nen for stimulating discussions concerning the technical details of the paper, and an anonymous referee for useful suggestions concerning the presentation of the results.




\vspace{2\baselineskip}

\noindent\textbf{References}

\bibliographystyle{acm}

\clearpage

\end{document}